%% file: Journal_nrel.tex
\documentclass{IEEEtran}
 
\IEEEoverridecommandlockouts                              
\usepackage{epsf}
\usepackage{epsfig}
\usepackage{times}
\usepackage{amssymb}
\usepackage{amsmath}
\usepackage{amsfonts}
\usepackage{latexsym}
\usepackage{url}
\usepackage{mathrsfs}
\usepackage{algorithm}
\usepackage{caption}
\usepackage{datetime}
\usepackage{bm}
\usepackage{algorithm,algorithmic}
\usepackage{subfigure}
\usepackage{framed} 
\usepackage[framed]{ntheorem}
\usepackage{cite}
\usepackage{amstext}  
\usepackage{array}      
\usepackage{mathabx}
\usepackage{color}
\usepackage{soul}

\newframedtheorem{frm-definition}{Definition}

\newtheorem{theorem}{Theorem}

\newtheorem{lemma}{Lemma}
\newtheorem{corollary}{Corollary}

\newtheorem{remark}{Remark}
\newtheorem{assumption}{Assumption}
\newtheorem{proposition}{Proposition}

\newcommand{\hL}{\mathcal{L}}
\newcommand{\hN}{\mathcal{N}}

\newcommand{\cD}{{\cal D}}

\newcommand{\cL}{{\cal{L}}}
\newcommand{\cN}{{\cal N}}

\newcommand{\cT}{{\cal T}}

\newcommand{\cE}{{\cal E}}

\newcommand{\cZ}{{\cal Z}}

\begin{document}
\title{Online Stochastic Optimization of Networked Distributed Energy Resources}
\author{Xinyang Zhou, Emiliano Dall'Anese and Lijun Chen
\thanks{X. Zhou is with the National Renewable Energy Laboratory, Golden, CO 80401, USA (Email: xinyang.zhou@nrel.gov).} 
\thanks{E. Dall'Anese and L. Chen are with the College of Engineering and Applied Science, University of Colorado, Boulder, CO 80309, USA  (Emails: \{emiliano.dallanese, lijun.chen\}@colorado.edu).}
\thanks{Preliminary results of this paper have been presented at IEEE International Conference on Smart Grid Communications (SmartGridComm), Dresden, German, October 2017 \cite{zhou2017stochastic}. }}

\maketitle

\begin{abstract}
This paper investigates distributed control and incentive mechanisms to coordinate distributed energy resources (DERs) with both continuous and discrete decision variables as well as device dynamics in distribution grids. We formulate a multi-period social welfare maximization problem, and based on its convex relaxation propose a distributed stochastic dual gradient algorithm for managing DERs. We further extend it to an online realtime setting with time-varying operating conditions, asynchronous updates by devices, and feedback being leveraged to account for nonlinear power flows as well as reduce communication overhead. The resulting algorithm provides a general online stochastic optimization algorithm for coordinating networked DERs with discrete power setpoints and dynamics to meet operational and economic objectives and constraints.  We characterize the convergence of the algorithm analytically and evaluate its performance numerically. 

%
\end{abstract}


\input{introduction_v2}


\section{System Model and Problem Formulation}\label{sec:model}

\subsection{Network Model}

Consider a distribution network with $N+1$ nodes collected in the set $\cN \cup \{0\}$, $\cN:=\{1, ..., N\}$, and distribution lines represented by the set $\cE$. Let $p_i^t \in \mathbb{R}$ and $q_i^t \in \mathbb{R}$ denote the aggregate net active  and reactive power injections, respectively,  at node $i \in \cN$ at time $t$.  Further, let $y^t$ be a vector collecting certain electrical quantities of interest; for example, voltage magnitudes at some selected nodes, current on some lines, or power flows at the substation.  The electrical quantities collected in $y^t$ are related to $p_i^t$ and $q_i^t$ via a nonlinear relationship that follows from Ohm's and Kirchhoff's Laws. In this paper, we utilize the following approximate linear relationship:  
\begin{eqnarray}\label{eq:linearization}
y^t \approx\hat{y}^t=Ap^t+Bq^t+c,
\end{eqnarray}
where $A,B$ and $c$ are linearization parameters that can be computed as shown in, e.g.,~\cite{christakou2013efficient, linModels} and pertinent references therein. It is worth pointing out that the linearized model~\eqref{eq:linearization} is utilized to facilitate the design of computationally-light algorithms that admit a real-time implementation. In Section~\ref{sec:online}, we will show how to leverage  appropriate measurements from the distribution grid and DERs to cope with the inaccuracies in the representation of the AC power flows and we will establish appropriate convergence claims. As shown in~\cite{christakou2013efficient, linModels}, the linear model~\eqref{eq:linearization} can be built in the way to account for an unbalanced system operation and for both wye and delta connections.

\subsection{Node \& Device Models}

At node $i\in\cN$, assume that non-controllable devices and controllable devices contribute to the aggregate net power injections $p_i^t \in \mathbb{R}$ and $q_i^t \in \mathbb{R}$. Denote by $p_{i,0}^t$ and $q_{i,0}^{t}$ the overall active and reactive powers injected by all the \emph{non-controllable} devices at node $i$. On the other hand,  a number of controllable devices are collected in a set $\cD_i$. Denote by $p^t_{i,d} \in \mathbb{R}$ and $q^t_{i,d} \in \mathbb{R}$ the active and reactive power injections of a controllable device $d\in\cD_i$ at time $t$. It follows that the powers at node $i$ can be expressed as:
	\begin{eqnarray}
	p_{i}^t=p_{i,0}^{t}+\sum_{d\in\cD_i}p^t_{i,d}, \ \ q^t_{i} =q_{i,0}^{t}+\sum_{d\in\cD_i}q^t_{i,d}.\label{eq:aggregator}
	\end{eqnarray} 

Power injection/consumption $z_{i,d}^t:=[p_{i,d}^t, q_{i,d}^t]^{\top}$ of device $d$ at time $t$ is constrained by  the set $\cZ^t_{i,d} \subset \mathbb{R}^2$ of available power setpoints, i.e.,  $z_{i,d}^t \in \cZ^t_{i,d}$. Denote by $W_t^w = \{t, t+1, \cdots, t+w\}$ the time window from time $t$ up to time $t+w$, and $\bm{z}_{i,d}^t:=[(z_{i,d}^{t})^\top,\ldots,(z_{i,d}^{t+w})^\top]^{\top}$ the power trajectory and $\bm{\cZ}^t_{{i,d}}$
the corresponding feasible set of device $d$ within this time window. 
We consider the following two types of controllable devices. 

\vspace{.1cm}

\noindent \textbf{Devices with convex sets:}  Devices with power injections/consumptions that can be chosen from a \emph{convex and compact} set. These devices are assumed to be fast responding, in the sense  that they regulate the power output to given commands within seconds. We collect these devices in the set $\cD_{F_i}\subseteq\cD_i$. 
For example, the feasible region of a  PV inverter $d\in\cD_{F_i}$ has the following form:
\begin{align} 
\cZ_{i,d}^t =  \left\{ z^t_{i,d} \big|  0 \leq p^t_{i,d}  \leq  p_{i,d}^{\text{av},t},~p_{i,d}^{t2} + q_{i,d}^{t2} \leq  \eta_{i,d}^2 \right\},\label{eq:pv}
\end{align}
where $p_{i,d}^{\text{av},t}$ denotes the available active power from a PV system at time $t$ (based on prevailing ambient conditions), and  $\eta_{i,d}$ is the rated apparent capacity. 
{In this case, we have $\bm{\cZ}_{i,d}^t = \bigtimes_{\tau=t}^{t+w}\cZ_{i,d}^{\tau}$, where $\bigtimes$ denotes Cartesian product.}

\vspace{.1cm}

\noindent \textbf{Devices with discrete sets:}  Devices that admit a {\em discrete set} of possible setpoints. Control actions of these devices are usually implemented at a slower timescale. Collect these devices in the set $\cD_{S_i}\subseteq\cD_i$, and denote by ${\cZ^t_{i,d}=} \mathbb{P}_{i,d}$ the set of power setpoints of a device $d\in\cD_{S_i}$; i.e., 
\begin{eqnarray}
z_{i,d}^t\in \mathbb{P}_{i,d}.\label{eq:discretechoice}
\end{eqnarray}

The devices may also feature states with given  dynamics. For a device $d$, let $\bm{x}_{i,d}^t:=[{x}_{i,d}^t,\ldots,{x}_{i,d}^{t+w}]^{\top}$ collect the states of the device from the current time $t$ up to  time $t+w$ (i.e., within the time window $W_t^w$). 
 We postulate an affine relationship between $\bm{x}_{i,d}^t$ and $\bm{z}_{i,d}^t$; i.e.,
\begin{eqnarray}
\bm{x}_{i,d}^{t}=\bm{f}^t_{i,d}(\bm{z}_{i,d}^{t}), \label{eq:evo}
\end{eqnarray}
{ where  $\bm{f}^t_{i,d}(\bm{z}_{i,d}^{t}):=[{f}^t_{i,d}(z_{i,d}^{t}),\ldots,{f}^{t+w}_{i,d}(\bm{z}_{i,d}^{t})]^{\top}$, with each $f_{i,d}^{\tau}(z_{i,d}^{t},\ldots,z_{i,d}^{\tau})$ being an affine function of power injections from $t$ to $\tau$ for all $\tau\in W_t^w$. This affine relation assumes that  the device dynamics are linear (or approximately linear) with power injection vector. We will provide two illustrative examples shortly.}
On the other hand, constraints on the state are modeled via a general \emph{convex} vector-valued function $\bm{\psi}_{i,d}^t$:
\begin{eqnarray}
\bm{\psi}_{i,d}^t(\bm{x}_{i,d}^{t}) \leq 0.\label{eq:statebound}
\end{eqnarray}
By substituting \eqref{eq:evo} into \eqref{eq:statebound}, we obtain the following constraint: 
\begin{eqnarray}
\bm{\psi}_{i,d}^t(\bm{f}^t_{i,d}(\bm{z}_{i,d}^{t})) \leq 0.\label{eq:statebound0}
\end{eqnarray}
{In this case, we have $\bm{\cZ}_{i,d}^t = \{\bm{z}_{i,d}^t |~\bm{z}_{i,d}^t\in \bigtimes_{\tau=t}^{t+w}\cZ_{i,d}^{\tau}~\&~\text{\eqref{eq:statebound0}}\}$.}

The set $\bm{\cZ}^t_{{i,d}}$ is convex if $d\in\cD_{F_i}$; on the other hand, it is \emph{discrete} and \emph{non-convex} if $d\in\cD_{S_i}$.

In the following, we provide two examples.

\noindent~~~~\emph{a)} For a heating, ventilation, and air conditioning (HVAC) system $d\in\cD_{S_i}$, let $x_{i,d}^{t}$ be the room temperature (i.e., the state) to be controlled at time $t$. The dynamics (\ref{eq:evo}) are, in this case,  in the form:
\begin{eqnarray}
x_{i,d}^{t+m}&\hspace{-7pt}=&\hspace{-7pt}T_{0,i,d}^{t+m}\!+\!\!\!\sum_{\tau=0}^{m-1}(1-\zeta_1)^{m-1-\tau}
\zeta_2 p_{i,d}^{t+\tau}.\label{eq:tp1}
\end{eqnarray}
Here, $T_{0,i,d}^{t+m}$ is a constant characterized as
\begin{eqnarray}
	\hspace{-3mm}T_{0,i,d}^{t+m}&\hspace{-7pt}=&\hspace{-7pt}(1-\zeta_1)^m x_{i,d}^{t}+\sum_{\tau=0}^{m-1}(1-\zeta_1)^{m-1-\tau}
	\zeta_2 T^{t+\tau}_{out,i,d} , 
\end{eqnarray}
where $T^{t+\tau}_{out,i,d}$ is the ambient temperature, and $\zeta_1, \zeta_2$ are parameters specifying thermal characteristics of the HVAC and its operating environment; see, e.g., \cite{li2011optimal}. The constraint (\ref{eq:statebound}) can be 
\begin{eqnarray}
\underline{\bm{T}}_{i,d}^{t}\leq \bm{x}_{i,d}^{t}\leq \overline{\bm{T}}_{i,d}^{t},\label{eq:tp3}
\end{eqnarray}
which is to confine the room temperatures  within a given comfort range $[\underline{\bm{T}}_{i,d}^{t},\overline{\bm{T}}_{i,d}^{t}]$. We  assume that the HVAC system can  be turned on/off; that is, 
\begin{eqnarray}
p_{i,d}^t\in\{p_{i,d}^{\text{on}},0\},~\forall t,\label{eq:ACpower}
\end{eqnarray}
with working power rate $p_{i,d}^{\text{on}}<0$. Then, the feasible set from $t$ to $t+w$ is $\bm{\cZ}^t_{i,d}=\{\bm{z}^t_{i,d}\,|\,\text{\eqref{eq:tp1}--\eqref{eq:ACpower}}\}$.

\noindent~~~~\emph{b)} For a {battery} 
$d\in\cD_{S_i}$, $x_{i,d}^t$ represents its state of charge (SOC)  at time $t$.   The dynamics (\ref{eq:evo}) for the evolution of SOC takes the form\footnote{Here we have assumed that the charging and discharging efficiencies are 1. 
We can deal with arbitrary battery efficiencies in the same way as in \cite{stai2018dispatching} that approximates the losses using an extended model; however, for simplicity of exposition, we consider an efficiency of 1. 
}
\begin{eqnarray}
x_{i,d}^{t+m}=x_{i,d}^t+\sum_{\tau=0}^{m-1} p_{i,d}^{t+\tau}, \label{eq:e1}
\end{eqnarray}
for $m=1,\ldots,w$. The constraint (\ref{eq:statebound}) in this case takes the following form:
\begin{eqnarray}
\underline{\bm{E}}_{i,d}^{t}\leq \bm{x}_{i,d}^{t}\leq \overline{\bm{E}}_{i,d}^t,\label{eq:e2}
\end{eqnarray}
with an acceptable SOC range of $[\underline{\bm{E}}_{i,d}^{t}, \overline{\bm{E}}_{i,d}^t]$. Further, a discrete set of available charging rates can be written as: 
\begin{eqnarray}
p_{i,d}^t\in\{p_{i,d}^{\text{dis},1},\ldots, p_{i,d}^{\text{dis},M},0, p_{i,d}^{\text{cha}, 1},\ldots, p_{i,d}^{\text{cha}, N}\},~\forall t, \label{eq:EVpower}
\end{eqnarray}
{where a negative charging rate means discharging.} The EV battery's feasible set from $t$ to $t+w$ is then $\bm{\cZ}^t_{i,d}=\{\bm{z}^t_{i,d}|\text{\eqref{eq:e1}--\eqref{eq:EVpower}}\}$.





\subsection{Problem Formulation}
\label{sec:formulation}
We aim to design an algorithmic strategy where  network operator and the DER-owners pursue their own operational goals and economic objectives, while concurrently achieving global coordination to enforce engineering constraints. For simplicity of exposition,  consider the following notation:
	\begin{align*}
	 &\bm{\cZ}^t_{i}:=\underset{d\in\cD_{i}}{\bigtimes}\bm{\cZ}^t_{i,d},&&\hspace{-15mm}\bm{\cZ}^t:=\underset{i\in\cN}{\bigtimes}\bm{\cZ}^t_{i},
	 \end{align*}
	\begin{align*}
	&\bm{z}^t_i= \{\bm{z}^t_{i,d}\}_{d\in\cD_i}\in\bm{\cZ}^t_{i},&&\hspace{-10mm} \bm{z}^t = \{\bm{z}^t_i\}_{i\in\cN}\in\bm{\cZ}^t.
	\end{align*}
We now describe pertinent optimization problems that model operational goals and constraints of network operator and the DER-owners. For simplicity of exposition, we outline the problem formulation under the presumption that one customer/DER-owner is located at a node of the electrical network; however, the proposed methodology is applicable to the case where multiple customers are located at a node (this is the case, for example, where multiple houses are connected to the same distribution transformer). 

\subsubsection{\textbf{Customer-Level Problem}}
Let  $\bm{\alpha}^{t}_i=[\alpha_i^t,\ldots,\alpha_i^{t+w}]^{\top}\in\mathbb{R}^{w+1}$ and $\bm{\beta}^{t}_i=[\beta_i^t,\ldots,\beta_i^{t+w}]^{\top}\in\mathbb{R}^{w+1}$ be vectors collecting ``incentive signals'' that are sent by the network operator to customer $i$ within the time window $W_t^w$ for real and reactive power, respectively. Let $x_i^t:=\{x_{i,d}^t\}_{d\in\cD_i}$ be a vector collecting  states of all devices of node $i$ at time $t$, and consider a cost function $C^t_{i}(\bm{z}^{t}_i,\bm{x}_i^t)$ that captures well-defined performance objectives of all devices at node $i$ within the time window $W_t^w$. Since $\bm{x}_i^t$ is in fact a function of $\bm{z}^{t}_i$, we will henceforth write the cost as $C^t_{i}(\bm{z}^{t}_i)$. Notice further that the cost function captures objectives over the time window $W_t^w$, and it can be expanded as an additive form, e.g., 
\begin{eqnarray}
C^t_{i}(\bm{z}^{t}_i)=\sum_{\tau=t}^{t+w}\sum_{d\in\cD_i} cost^{\tau}_{i,d}(z^{\tau}_{i,d}),\nonumber
\end{eqnarray}
where $cost^{\tau}_{i,d}(z^{\tau}_{i,d})$ denotes the cost function for device $d$ at time $\tau$.

The discrete nature of the control actions associated with devices in the sets $\{\cD_{S_i}\}$ would render pertinent optimization problems nonconvex. {Consider then utilizing the convex hull of the set $\mathbb{P}_{i,d}$ of available power setpoints, defined as follows \cite{boyd2004convex}: 
\begin{eqnarray}
	&&conv(\mathbb{P}_{i,d}):=\nonumber\\
	&&\Big\{\sum_{m=1}^{|\mathbb{P}_{i,d}|}\theta^m_{i,d}\cdot p^{m}_{i,d}\Big|p^{m}_{i,d}\in\mathbb{P}_{i,d}, \theta^m_{i,d}\geq 0, \sum_{m=1}^{|\mathbb{P}_{i,d}|}\theta^m_{i,d}=1\Big\},\nonumber
\end{eqnarray}
where $|\mathbb{P}_{i,d}|$ denotes the cardinality of the set $\mathbb{P}_{i,d}$. With a slight abuse of notation, we next define
\begin{eqnarray}
conv(\bm{\mathcal{Z}}^t_{{i,d}}) := \big\{\bm{z}_{{i,d}}^t|~\text{\eqref{eq:statebound0}}~\&~z_{{i,d}}^{\tau} \in conv(\mathbb{P}_{i,d}),\nonumber\\
 \tau=t,\ldots,t+w \big\}\label{eq:rels}
\end{eqnarray}
as the relaxed set of $\bm{\mathcal{Z}}^t_{{i,d}}$.} The relaxed set $conv(\bm{\mathcal{Z}}^t_{i,d})$ is \emph{convex and compact}. Accordingly, consider the following convex sets:
\begin{eqnarray}
conv(\bm{\mathcal{Z}}_i^t):=\bigg(\underset{d\in\cD_{F_i}}{\bigtimes} \bm{\mathcal{Z}}^t_{i,d} \bigg)\bigtimes\bigg(\underset{d\in\cD_{S_i}}{\bigtimes} conv(\bm{\mathcal{Z}}^t_{i,d})\bigg),\nonumber
\end{eqnarray}
and
\begin{eqnarray}
conv(\bm{\mathcal{Z}}^t):=\underset{i\in\cN}{\bigtimes} conv(\bm{\mathcal{Z}}_i^t).
\end{eqnarray}
While a convex hull {of the set of available power setpoints} is utilized for the algorithmic design, a randomization strategy will be leveraged to recover {these available} control actions from $conv(\bm{\mathcal{Z}}^t_{{i,d}})$. 


For given vectors $(\bm{\alpha}^{t}_i,\bm{\beta}^{t}_i)$, and for a given optimization horizon of $w+1$ time slots, the following problem is solved at node $i$ at time $t$:
\begin{subequations}\label{eq:p1}
	\begin{eqnarray}
	&\hspace{-6mm}(\bm{\mathcal{P}^{t}_{1,i}}) \,\,\, \underset{\bm{z}_i^t}{\min} & C^{t}_{i}(\bm{z}^{t}_i)-\sum_{\tau=t}^{t+w}\sum_{d\in\cD_i}({\alpha}^{\tau}_i {p}^{\tau}_{i,d}+{{\beta}_i^{\tau}} {q}^{\tau}_{i,d}),\label{eq:obj_c}\\
	&\hspace{-6mm}\mathrm{s.t.}&\bm{z}_{i}^{t} \in  conv(\bm{\cZ}^{t}_i), \label{eq:pccon_c}
	\end{eqnarray}
\end{subequations}
where the terms $\alpha^{\tau}_i p^{\tau}_{i,d}$ and $\beta^{\tau}_i q^{\tau}_{i,d}$ are utilized to enable responsiveness to the network-level signals $(\bm{\alpha}^{t}_i,\bm{\beta}^{t}_i)$ (which can be interpreted as incentives or prices), and device-specific states are naturally accounted for in constraint \eqref{eq:pccon_c}. The following assumption is imposed.

\begin{assumption}\label{as:Cinv}
		Functions $C^t_{i}(\bm{z}^{t}_i)$ for all $i\in\cN$ are strongly convex in $\bm{z}^{t}_i\in conv(\bm{\cZ}^t_{i})$; i.e., there exists some $\sigma_c>0$ such that $\nabla^2_{\bm{z}_i}C^t_{i}(\bm{z}^{t}_i)\geq\sigma_c I$.
		Moreover, the gradient of $C^t_{i}(\bm{z}^{t}_i)$ with respect to $\bm{z}^{t}_i$ is Lipschitz continuous for any $\bm{z}^{t}_i\in conv(\bm{\cZ}^t_{i})$.
\end{assumption}

%

{
\begin{remark}
	Assumption~\ref{as:Cinv}  is standard, and it is utilized here to facilitate the derivation of the convergence results; see, e.g., \cite{BeT89,boyd2004convex,zhou2017incentive2,dallanese2016optimal}. In practice, quadratic cost functions or piece-wise quadratic cost functions are often used in the market models. These functions satisfy Assumption~\ref{as:Cinv}. Moreover, if the cost functions associated with some DERs are not strongly convex, we can add a regularization term to obtain a strongly convex cost. 
	\end{remark}
}

Since (\ref{eq:obj_c}) is strongly convex  and that the constraints are convex, 
a unique solution $\bm{z}_i^{t*}$ exists for given $\bm{\alpha}_i^t$ and $\bm{\beta}_i^t$, represented as the following mapping: 
\begin{eqnarray}
\bm{z}_i^{t*}=\bm{b}_i^t(\bm{\alpha}^t_i, \bm{\beta}^t_i).
\end{eqnarray}
{
Given Assumption~\ref{as:Cinv}, function $\bm{b}_i^t$ is usually continuous. For example, $\bm{b}_i^t$ is piece-wise linear if $C^t_{i}(\bm{z}^{t}_i)$ is quadratic and $conv(\bm{\cZ}^{t}_i)$ is a box constraint. Nevertheless, there is no requirement on the continuity of $\bm{b}_i^t$. As will be seen in Section~\ref{sec:off}, solving for $b_i^t$ is not required by considering a  convex relaxation (which will turn out to be exact).
}
%

{
\begin{remark}
Notice that the relaxed set $conv(\bm{\mathcal{Z}}^t_{{i,d}})$ defined in Eq.~\eqref{eq:rels} is not the convex hull of the non-convex set $\bm{\mathcal{Z}}^t_{{i,d}}$. Given a non-convex set (or non-convex constraints), there may be different ways to convexify it. Different convexification schemes may incur different computational and description complexities (as well as result in different numerical stabilities). In our convexification scheme, we only convexify those non-convex constraints, i.e., Eq.~\eqref{eq:discretechoice}. Such a scheme has minimum complexity for our problem. 
\end{remark}
}

\subsubsection{\textbf{Social-Welfare Problem}} Let $g^t(y^t)$ be an affine function of the electrical quantities of interest $y^t$, and let the following constraints capture network-level operational constraints at time $t$:
\begin{eqnarray}
g^t(\hat{y}^t)\leq 0 \, .\label{syscons}
\end{eqnarray}
By Eqs.~\eqref{eq:linearization}--\eqref{eq:aggregator}, $\hat{y}^t$ is  affine in $z^t$. For future development, define the vector-valued function $\bm{g}^t(\bm{z}^{t}):=[g^t(\hat{y}^t(z^{t}))^{\top},\ldots,g^{t+w}(\hat{y}^{t+w}(z^{t+w}))^{\top}]^{\top}$ within the time window $W_t^w$,  which is affine in $\bm{z}^{t}$. Notice that the operational constraints are usually ``independent'' in the sense that none of them will subsume any other. We therefore have the following assumption.
{\begin{assumption}\label{ass:g}
Function  $\bm{g}^t(\bm{z}^{t})$ is an affine function of $\bm{z}^{t}$.
\end{assumption}}

With Assumption~\ref{ass:g} and the boundedness of the set $conv(\bm{\cZ}^t)$, the Jacobian of $\bm{g}^t(\bm{z}^t)$ is bounded over $conv(\bm{\cZ}^t)$; i.e., there exists some constant $\sigma_g>0$ such that 
\begin{eqnarray}
\|\nabla_{\bm{z}} \bm{g}^t(\bm{z}^t)\|_F\leq \sigma_g,\ \forall \bm{z}^t\in conv(\bm{\cZ}^t),
\end{eqnarray}
where $\|\cdot\|_F$ denotes the Frobenius norm.

With these definitions in place, the following optimization problem is formulated to minimize the aggregate cost incurred by the customers, subject to network constraints: 
\begin{subequations}
\begin{eqnarray}
\hspace{-6mm} (\bm{\mathcal{P}^{t}_2})& \min&\sum_{i\in\cN}C^{t}_{i}(\bm{z}_i^{t})\label{eq:obj}\\
 \hspace{-6mm}&\mathrm{over}& \bm{z}^t,\hat{\bm{y}}^t,\bm{\alpha}^t,\bm{\beta}^t\\
\hspace{-6mm}&\mathrm{s.t.} & \bm{p}^{t}_i=\bm{p}_{i,0}^{t}\!+\!\!\!\sum_{d\in\cD_i}\bm{p}^{t}_{i,d},\ i\in\cN,\\[-2pt]
\hspace{-6mm}&&\bm{q}^{t}_i=\bm{q}_{i,0}^{t}\!+\!\!\!\sum_{d\in\cD_i}\bm{q}^{t}_{i,d},\ i\in\cN,\\[-2pt]
\hspace{-6mm} &&\hat{y}^{\tau}= Ap^{\tau} + Bq^{\tau} + c,\ \tau\in W_t^w,\label{eq:volt0}\\
\hspace{-6mm} &&\bm{g}^{t}(\hat{\bm{y}}^{t})\leq 0, \label{eq:volt}\\
\hspace{-6mm} && \bm{z}_i^{t}=\bm{b}^t_{i}(\bm{\alpha}^{t}_i,\bm{\beta}^{t}_i),\ \forall i\in\hN,\label{eq:pccon}
\end{eqnarray}
\end{subequations}
where constraint \eqref{eq:pccon} explicitly accounts for the  solution of the customer-level problem $(\bm{\mathcal{P}^{t}_{1,i}})$. Notice that $\bm{\alpha}^t$ and $\bm{\beta}^t$ are optimization variables, and they are designed in a way that customer responses \eqref{eq:pccon} do not lead to violation of network constraints. 

Unfortunately, constraint \eqref{eq:pccon} renders problem $(\bm{\mathcal{P}^t_2})$ nonconvex as it is usually not affine. 
Before outlining the solution approach to address nonconvexity brought in by \eqref{eq:pccon} (see Section~\ref{sec:exactrelax}), we briefly explain how to recover discrete control actions for the devices $\cD_{S_i}$ next.

\subsubsection{\textbf{Recovering Device-Specific Feasible Power Setpoints}}
For a device $d \in \cD_{S_i}$ with discrete power commands,\footnote{Since we have a mix of DERs with discrete commands and continuous commands, we use $p_{i,d}^t$ instead of $z_{i,d}^t$ when considering recovered feasible power setpoints. Similar arguments apply to the reactive power.} the solution $\bm{p}_{i,d}^{t*}$ of the problem $(\bm{\mathcal{P}^{t}_{1,i}})$ may not be  implementable, since it is computed based on the convex hull of the discrete set $\mathbb{P}^t_{i,d}$. To recover a feasible  command implementable { at device $d$} at time $t$,  notice  that $p^{t*}_{i,d}\in conv(\mathbb{P}^t_{i,d})$ can be written as a convex combination $\sum_{m=1}^{|\mathbb{P}^t_{i,d}|} \theta^m_{i,d} p^{tm}_{i,d}$ of certain points $p^{tm}_{i,d}$ in $\mathbb{P}^t_{i,d}$, where ${\theta^m_{i,d} \geq 0}$ are such that $\sum_{m=1}^{|\mathbb{P}^t_{i,d}|}  \theta^m_{i,d}=1$. Then,  a feasible power command  $p^{tm}_{i,d}\in\mathbb{P}^t_{i,d}$ can be randomly selected with probability $\theta^m_{i,d}$.

As an illustrative example, assume that a device has two feasible power commands denoted by ${p}^{t1}_{i,d}$ and ${p}^{t2}_{i,d}$. It follows that ${p}^{t1}_{i,d}\leq p^{t*}_{i,d}\leq {p}^{t2}_{i,d}$. 
{Then, we can select  ${p}^{t1}_{i,d}$ and ${p}^{t2}_{i,d}$ with the following two-point probability distribution:
\begin{eqnarray}\label{eq:probability}
\left \{ 
\begin{array}{l l l }
\text{Prob}(p^t_{i,d}={p}^{t1}_{i,d})&=&({{p}^{t2}_{i,d}-p_{i,d}^{t*}})/({{p}^{t2}_{i,d}-{p}^{t1}_{i,d}})\\
\text{Prob}(p^t_{i,d}={p}^{t2}_{i,d})&=&({p_{i,d}^{t*}-{p}^{t1}_{i,d}})/({{p}^{t2}_{i,d}-{p}^{t1}_{i,d}})
\end{array}\right..
\end{eqnarray}
}

	\begin{remark}\label{rem:imp} The above scheme to recover the device-specific feasible power setpoints requires finding the convex combination coefficients $\theta^m_{i,d}$. The complexity of finding these coefficients is linear in the number of discrete setpoints in $\mathbb{P}^t_{i,d}$ of device $d$, which is relatively small. Since the power setpoint is a scalar for each device, a simple procedure to find these coefficients is to identify the largest setpoint below $p_{i,d}^{t*}$ and the smallest setpoint above $p_{i,d}^{t*}$, and calculate the coefficients corresponding to these two setpoints similar to Eq.~\eqref{eq:probability}, and the coefficients corresponding to other setpoints are zero. Notice that the set of convex combination coefficients may not be unique. 
	\end{remark}

\section{Distributed Stochastic Dual Algorithm}\label{sec:off}

At time $t$,  problem $(\bm{\mathcal{P}^t_2})$ naturally leads to a Stackelberg game where: (i) $\bm{\alpha}^t$ and $\bm{\beta}^t$ are calculated via $(\bm{\mathcal{P}^t_2})$  by the network operator and broadcasted to all nodes $i\in\cN$; and, (ii) each consumer computes the power set points $\bm{z}_i^{t*}$ from $(\bm{\mathcal{P}^t_{1,i}})$. By design, $\bm{z}^{t*}$ is in fact an optimal point for $(\bm{\mathcal{P}^t_2})$.  

In the following, we first recall a result presented in our prior work \cite{zhou2017incentive2} to obtain an exact convex relaxation of $(\bm{\mathcal{P}^t_2})$ regarding non-convex constraint \eqref{eq:pccon}. Then, Section~\ref{sec:iter} and Section~\ref{sec:distributed} will present a  new stochastic dual algorithm to identify optimal points of $(\bm{\mathcal{P}^t_2})$ and to recover feasible power commands for devices in $\{ \cD_{S_i}\}$.

\subsection{Exact Convex Relaxation}\label{sec:exactrelax}
Consider the following convex optimization problem:
\begin{subequations}
\begin{eqnarray}
\hspace{-6mm} (\bm{\mathcal{P}^{t}_3})&\min&\sum_{i\in\cN}C^{t}_{i}(\bm{z}_i^{t})\label{eq:obj2}\\
 \hspace{-6mm}&\mathrm{over}& \bm{z}^t, \hat{\bm{y}}^t\\
\hspace{-6mm}&\mathrm{s.t.} & \bm{p}^{t}_i=\bm{p}_{i,0}^{t}\!+\!\!\!\sum_{d\in\cD_i}\bm{p}^{t}_{i,d},\ i\in\cN,\label{eq:aggregate}\\[-2pt]
\hspace{-6mm}&&\bm{q}^{t}_i=\bm{q}_{i,0}^{t}\!+\!\!\!\sum_{d\in\cD_i}\bm{q}^{t}_{i,d},\ i\in\cN,\\[-2pt]
\hspace{-6mm}&&\hat{y}^{\tau}= Ap^{\tau} + Bq^{\tau} + c,\ \tau\in W_t^w,\label{eq:volt02}\\
\hspace{-6mm}&&\bm{g}^t(\hat{\bm{y}}^{t}) \leq 0,\label{eq:volt2}\\
\hspace{-6mm}&&\bm{z}_i^{t}\in conv(\bm{\cZ}^{t}_i),\ i\in\hN,\label{eq:pccon_c2}
\end{eqnarray}
\end{subequations}
which is obtained from $(\bm{\mathcal{P}^t_2})$ by dropping the nonconvex constraint (\ref{eq:pccon}) and instead adding the device-specific constraints (\ref{eq:pccon_c2}). The following is assumed. 

\begin{assumption}\label{ass:sla}
Problem $(\bm{\mathcal{P}_3^{t}})$ is strictly feasible, i.e.,  it satisfies Slater's condition at each time $t$. 
\end{assumption}

{Assumption~\ref{ass:sla} ensures that strong duality holds for $(\bm{\mathcal{P}^{t}_3})$ \cite{boyd2004convex}.} Given the strong convexity of the cost function  \eqref{eq:obj2} in $\bm{z}^t$, a unique optimal solution exists for problem $(\bm{\mathcal{P}^t_3})$. Leveraging the results of \cite{zhou2017incentive2}, we can show that $(\bm{\mathcal{P}^t_3})$ together with a particular choice of $\bm{\alpha}^t$ and $\bm{\beta}^t$ lead to an exact convex relaxation of $(\bm{\mathcal{P}^t_2})$; i.e., 
the solution of $(\bm{\mathcal{P}^t_3})$ coincides with an optimal solution of $(\bm{\mathcal{P}^t_2})$. Specifically, substitute Eqs.~(\ref{eq:aggregate})--(\ref{eq:volt02}) into (\ref{eq:volt2}), and denote by $\mu^{\tau}$ the dual variable  associated with the constraints~\eqref{eq:volt2} for a given  $\tau$. Let $\hat{\bm{y}}^{t*}$ be the optimal solution of $(\bm{\mathcal{P}^t_3})$, and denote the optimal dual variables associated with~\eqref{eq:volt2} as ${\mu}^{\tau*}$ for time $\tau$. Consider then the following choice for $\bm{\alpha}^t$ and $\bm{\beta}^t$:
\begin{subequations}\label{eq:signal}
	\begin{eqnarray}
		\alpha^{\tau*}&=& -A^{\top}\nabla_{\hat{y}} g^{\tau}(\hat{y}^{\tau*}) \mu^{\tau*},\\
		\beta^{\tau*}&=& -B^{\top}\nabla_{\hat{y}} g^{\tau}(\hat{y}^{\tau*})\mu^{\tau*},
	\end{eqnarray}
\end{subequations}
for all $\tau \in W_t^w$.
Then, Theorems~1--2 in \cite{zhou2017incentive2} can be extended to obtain the following result. 

\begin{theorem}\label{the1}
Under Assumptions~\ref{as:Cinv}--\ref{ass:sla}, it follows that $(\bm{\alpha}^{t*}, \bm{\beta}^{t*})$ defined by \eqref{eq:signal} are bounded. Moreover, the solution of problem $(\bm{\mathcal{P}^t_3})$ along with $(\bm{\alpha}^{t*}, \bm{\beta}^{t*})$ defined in (\ref{eq:signal}) is a globally optimal solution of the nonconvex problem $(\bm{\mathcal{P}^t_2})$.
\end{theorem}

Hereafter, we will refer to the globally optimal solutions of $(\bm{\mathcal{P}^t_3})$ and $(\bm{\mathcal{P}^t_2})$ interchangeably, depending on the context. In the next section, we will  design a stochastic dual algorithm for solving $(\bm{\mathcal{P}^t_3})$ in an offline or batch setting, and show how to  recover feasible commands for  devices with discrete power levels. Subsequently,  we will develop an online stochastic dual algorithm.

\begin{remark}
From a practical standpoint, $\bm{\alpha}^t$ and $\bm{\beta}^t$ can be interpreted as ``incentive signals'' or ``prices'' that the network operator communicates to customers to ensure that engineering constraints in the network are satisfied. For example, when $y_t$ collects voltage magnitudes and the function $g$ is designed to ensure voltage regulation, $\bm{\alpha}^t$ and $\bm{\beta}^t$ can be understood  as prices of voltage violation. 
\end{remark}

\subsection{Offline Stochastic Dual Algorithm}
\label{sec:iter}

Consider the Lagrangian function associated with $(\bm{\mathcal{P}^t_3})$:
\vspace{-.0cm}
\begin{eqnarray}
L^t(\bm{z}^t,\bm{\mu}^t)=\sum_{i\in\cN}C^{t}_{i}(\bm{z}_i^{t})+\bm{\mu}^{t\top}\bm{g}^{t}(\bm{z}^t),  \label{eq:Lag}
\end{eqnarray}
where $\bm{\mu}^t$ is the vector of dual variables associated with the network constraints~\eqref{eq:volt2}. The Lagrangian~\eqref{eq:Lag} is obtained by substituting Eqs.~\eqref{eq:aggregate}--\eqref{eq:volt02} into \eqref{eq:volt2} and by keeping the constraints (\ref{eq:pccon_c2}) implicit. { To facilitate the design of online algorithms that exhibit linear convergence, we  consider the following regularized Lagrangian function:}
\vspace{-.0cm}
\begin{eqnarray}
\hL^t(\bm{z}^t,\bm{\mu}^t): =L^t(\bm{z}^t,\bm{\mu}^t) - \frac{\phi}{2}\|\bm{\mu}^t\|^2,  \label{eq:Lag_phi}
\end{eqnarray}
{where $\phi>0$ is a predefined parameter~\cite{koshal2011multiuser}. Notice that $\hL^t(\bm{z}^t,\bm{\mu}^t)$ is strongly convex in $\bm{z}^t$ and strongly concave in $\bm{\mu}^t$; therefore,  $\hL^t(\bm{z}^t,\bm{\mu}^t)$ has a unique saddle point. On the other had, the saddle point of $\hL^t(\bm{z}^t,\bm{\mu}^t)$ does not coincide with the primal-dual optimum of \eqref{eq:Lag}; it is, in fact, related to the concept of approximate Karush-Kuhn-Tucker point. Bounds on the distance between the saddle points of \eqref{eq:Lag} and \eqref{eq:Lag_phi} (as a function of $\phi$) can be found in, e.g.,~\cite{koshal2011multiuser}.}

Define the primal optimal solution at time $t$ as for a given dual $\bm{\mu}^t$: 
\begin{subequations}
\begin{eqnarray}
\bm{z}^{t*}&:=&\underset{\bm{z}^t}{\arg\min}~\hL^t(\bm{z}^t,\bm{\mu}^t),\\
&&\hspace{6mm}\mathrm{s.t.}~ ~~\text{(\ref{eq:pccon_c2})},
\end{eqnarray}
\end{subequations}
and consider the following 
dual problem
\begin{eqnarray} 
\underset{\bm{\mu}^t \geq 0}{\max}~~h^t(\bm{\mu}^t) \label{eq:dual}
\end{eqnarray}
where the dual function is defined as:
\begin{eqnarray}
h^t(\bm{\mu}^t)= \hL^t(\bm{z}^{t*},\bm{\mu}^t). 
\end{eqnarray}
We then have the following result. 

\begin{lemma}[Theorem~2.1 of \cite{necoara2014rate}]\label{lem:gradient}
Under Assumptions~\ref{as:Cinv}--\ref{ass:g}, the gradient of the dual function is given by $\nabla_{\bm{\mu}} h^t(\bm{\mu}^t)=\bm{g}^t(\bm{z}^{t*}){-\phi\bm{\mu}^t}$. Furthermore, the  gradient of the dual function is Lipschitz continuous with constant ${\sigma^2_g}/{\sigma_c}{+\phi}$, i.e., 
\begin{eqnarray}
\| \nabla_{\bm{\mu}} h^t(\bm{\mu})-\nabla_{\bm{\mu}} h^t(\tilde{\bm{\mu}})\|\leq  ({\sigma^2_g}/{\sigma_c}{+\phi}) \|\bm{\mu}-\tilde{\bm{\mu}}\| \label{eq:2ndlp}
\end{eqnarray}
for any feasible $\bm{\mu}$ and $\tilde{\bm{\mu}}$.
\end{lemma}

Using the results of Theorem~2.1 of \cite{necoara2014rate}, we augment a dual gradient method for solving (\ref{eq:dual}) with the a randomization strategy to recover discrete power commands at each iteration. The resultant stochastic dual gradient algorithm involves a sequential execution of the following steps:  
\begin{subequations}\label{eq:dualalg}
\begin{eqnarray}
&&\hspace{-7 mm} \textrm{[S1]} \, \bm{z}^{t*}(k+1)=\underset{\bm{z}^t}{\arg\min}~\hL^t(\bm{z}^t,\bm{\mu}^t(k)),~\mathrm{s.t.}~ \text{(\ref{eq:pccon_c2})},\label{eq:random}\\
&&\hspace{-7 mm} \textrm{[S2]} ~\text{For}~d\in\cD_{S_i}, \text{~pick}~{z}_{i,d}^{t}(k+1)\in \cZ^t_{i,d}~\text{randomly based on} \nonumber\\
&&z_{i,d}^{t*}(k+1)~\text{using scheme described in Section II.C.1b}, \nonumber\\
&&\hspace{-7 mm} \textrm{[S3]} \text{~Set}~{z}_{i,d}^{\tau}(k+1)={z}_{i,d}^{\tau*}(k+1),\forall d\in\cD_{S_i},\ \forall \tau\neq t,\nonumber\\
&&\hspace{-0mm}\text{~and}~\bm{z}_{i,d}^{t}(k+1)=\bm{z}_{i,d}^{t*}(k+1),\forall d\in\cD_{F_i},\label{eq:dualrandom}\\
&&\hspace{-7 mm} \textrm{[S4]} \, \bm{\mu}^{t}(k+1)=\big[\bm{\mu}^{t}(k)+\varepsilon \big(\bm{g}^t(\bm{z}^{t}(k+1)){-\phi\bm{\mu}^t(k)}\big)\big]_+~, \label{eq:dualalg3}
\end{eqnarray}
\end{subequations}
where $[~]_+$ denotes the projection onto the nonnegative orthant, and $\varepsilon>0$ is a given stepsize. 
Notice that \eqref{eq:dualrandom} recovers feasible power commands only for the current time $t$ (and not for $\tau = t+1, \ldots, t+w$). Further, in step~(\ref{eq:dualalg3}), $\bm{g}^t(\bm{z}^{t}(k+1)){-\phi\bm{\mu}^t(k)}$ is stochastic gradient of $h^t(\bm{\mu}^t(k))$ with $E[\nabla_{\bm{\mu}} h^t(\bm{\mu}^t(k))]=\bm{g}^t(\bm{z}^{t*}(k+1)){-\phi\bm{\mu}^t(k)}$ (the gradient).  

Let $\sigma_h>0$ be the strong concavity coefficient of $h^t(\bm{\mu}^t)$; that is, for any feasible $\bm{\mu}$ and $\tilde{\bm{\mu}}$, it holds that:
\begin{eqnarray}
\big(\nabla_{\bm{\mu}} h^t(\bm{\mu})-\nabla_{\bm{\mu}} h^t(\tilde{\bm{\mu}})\big)^{\top}(\bm{\mu}-\tilde{\bm{\mu}}) \leq -\sigma_h \|\bm{\mu}-\tilde{\bm{\mu}}\|^2.\label{eq:2ndconcave}
\end{eqnarray}

{
\begin{lemma}\label{lem:sigmarelation}
	The following relationship holds for the strong concavity parameter $\sigma_h$ and the Lipschitz continuity parameter $\sigma_g^2/\sigma_c{+\phi}$:
	\begin{eqnarray}
	\sigma_h \leq \sigma_g^2/\sigma_c{+\phi}.\label{eq:sigmarelation}
	\end{eqnarray}
\end{lemma}
\begin{proof}
	From \eqref{eq:2ndconcave} we have
	\begin{eqnarray}
	\big(-\nabla_{\bm{\mu}} h^t(\bm{\mu})+\nabla_{\bm{\mu}} h^t(\tilde{\bm{\mu}})\big)^{\top}(\bm{\mu}-\tilde{\bm{\mu}}) \geq \sigma_h \|\bm{\mu}-\tilde{\bm{\mu}}\|^2.\nonumber
	\end{eqnarray}
	On the other hand, the following holds:
	\begin{eqnarray}
	&&\big(-\nabla_{\bm{\mu}} h^t(\bm{\mu})+\nabla_{\bm{\mu}} h^t(\tilde{\bm{\mu}})\big)^{\top}(\bm{\mu}-\tilde{\bm{\mu}})\nonumber\\
	&\leq&\big\|-\nabla_{\bm{\mu}} h^t(\bm{\mu})+\nabla_{\bm{\mu}} h^t(\tilde{\bm{\mu}})\big\|\cdot\|\bm{\mu}-\tilde{\bm{\mu}}\|\nonumber\\
	&\leq&({\sigma^2_g}/{\sigma_c}{+\phi}) \|\bm{\mu}-\tilde{\bm{\mu}}\| \cdot\|\bm{\mu}-\tilde{\bm{\mu}}\|\nonumber
	\end{eqnarray}
	with the second inequality from \eqref{eq:2ndlp}. This leads to \eqref{eq:sigmarelation}.
\end{proof}
}

\subsection{Convergence Analysis}	

%
%
%
%
%
	
	Let $\Delta > 0$ be a given scalar such that the following holds uniformly in time: 
\begin{eqnarray}
\text{Var}\big(\bm{g}^t(\bm{z}^t)\big):=E[\|\bm{g}^t(\bm{z}^t)\|^2]-\|\bm{g}^t(\bm{z}^{t*})\|^2\leq \Delta, \label{eq:Delta}
\end{eqnarray}
%
where the variance is taken with respect to the randomization step~\eqref{eq:dualrandom}. In other words, $\text{Var}\big(\bm{g}^t(\bm{z}^t)\big)$ represent the variance of the discrepancy between the constraint function evaluated at the relaxed solution and at the randomized solution.  
{
Notice that, given a system with a finite number of devices with finite power commands, a finite $\Delta$ always exists.  See Appendix~\ref{app:Delta} for an illustrative example of characterizing $\Delta$ under the probability distribution \eqref{eq:probability}. }
 
The following convergence result for the dual variables can be stated.

\begin{theorem}\label{the:converge}
Under Assumptions~\ref{as:Cinv}--\ref{ass:sla}, and with a  stepsize $\varepsilon$ chosen such that
{\begin{eqnarray}
\varepsilon<2\sigma_h{/(\sigma^2_g/\sigma_c+\phi)^2},\label{eq:stepconverge}
\end{eqnarray}}
the stochastic dual algorithm~\eqref{eq:dualalg}
converges as 
\begin{eqnarray}
\lim_{k\rightarrow\infty}E[\|\bm{\mu}^t(k)-\bm{\mu}^{t*}\|^2]=\frac{\varepsilon\Delta}{2\sigma_h-\varepsilon{(\sigma^2_g/\sigma_c+\phi)^2}}.\label{eq:convergeresult}
\end{eqnarray}
\end{theorem}
 
%
\begin{proof}
Let $\bm{\mu}^{t*}$ be the vector of the unique optimal dual variables. Then,
\begin{eqnarray}
&&\hspace{-3mm}E\big[\|\bm{\mu}^t(k+1)-\bm{\mu}^{t*}\|^2\big|\bm{\mu}^t(k)\big]\nonumber\\
&\hspace{-3mm}\leq&\hspace{-3mm}E\big[\|\bm{\mu}^t(k)+\varepsilon \big(\bm{g}^t(\bm{z}^{t}(k+1)){-\phi\bm{\mu}^t(k)}\big)\nonumber\\
&\hspace{-3mm}&-\bm{\mu}^{t*}-\varepsilon \big(\bm{g}^t(\bm{z}^{t*}){-\phi\bm{\mu}^{t*}}\big)\|^2\big|\bm{\mu}^t(k)\big]\nonumber\\
&\hspace{-3mm}=&\hspace{-3mm}  \|\bm{\mu}^t(k)-\bm{\mu}^{t*}\|^2 + \varepsilon^2E\big[\|\bm{g}^t(\bm{z}^t(k+1)){-\phi\bm{\mu}^t(k)}-\bm{g}^t(\bm{z}^{t*})\nonumber\\
&&{+\phi\bm{\mu}^{t*}}\|^2\big|\bm{\mu}^t(k)\big] +2\varepsilon (\bm{\mu}^t(k)-\bm{\mu}^{t*})^{\top}(\bm{g}^t(\bm{z}^{t*}(k+1))\nonumber\\
&&{-\phi\bm{\mu}^t(k)}-\bm{g}(\bm{z}^{t*}){+\phi\bm{\mu}^{t*}})\nonumber\\
&\hspace{-3mm}=&\hspace{-3mm}  \|\bm{\mu}^t(k)-\bm{\mu}^{t*}\|^2 + \varepsilon^2E\big[\|\bm{g}^t(\bm{z}^t(k+1)){-\phi\bm{\mu}^t(k)}-\bm{g}^t(\bm{z}^{t*})\nonumber\\
&&{+\phi\bm{\mu}^{t*}}\|^2\big|\bm{\mu}^t(k)\big] +2\varepsilon (\bm{\mu}^t(k)-\bm{\mu}^{t*})^{\top}(\nabla h^t(\bm{\mu}^t(k))\nonumber\\
&&-\nabla h^t(\bm{\mu}^{t*}))\nonumber
\end{eqnarray}
\begin{eqnarray}
&\hspace{-3mm}\leq&\hspace{-3mm}  \|\bm{\mu}^t(k)-\bm{\mu}^{t*}\|^2 + \varepsilon^2E\big[\|\bm{g}^t(\bm{z}^t(k+1)){-\phi\bm{\mu}^t(k)}-\bm{g}^t(\bm{z}^{t*})\nonumber\\
&&{+\phi\bm{\mu}^{t*}}\|^2\big|\bm{\mu}^t(k)\big]-2\varepsilon\sigma_h\|\bm{\mu}^t(k)-\bm{\mu}^{t*}\|^2,\nonumber
\end{eqnarray}
where the first inequality follows from the non-expansiveness property of the projection operator; the first equality is due to the fact that $\bm{g}^t$ is linear and it accounts for the recovery of discrete solutions; the second equality follows from Lemma~\ref{lem:gradient}; and the last inequality is due to the strong concavity of $h^t(\bm{\mu}^t)$.
Notice that
\begin{eqnarray}
&\hspace{-3mm}&\hspace{-1mm}E\big[\|\bm{g}^t(\bm{z}^t(k+1)){-\phi\bm{\mu}^t(k)}-\bm{g}^t(\bm{z}^{t*}){+\phi\bm{\mu}^{t*}}\|^2\big|\bm{\mu}^t(k)\big]\nonumber\\
&\hspace{-3mm}= &\hspace{-1mm} Var\big(\bm{g}^t(\bm{z}^t(k+1))\big|\bm{\mu}^t(k)\big)\nonumber\\
&&+\|\bm{g}^t(\bm{z}^{t*}(k+1)){-\phi\bm{\mu}^t(k)}-\bm{g}^t(\bm{z}^{t*}){+\phi\bm{\mu}^{t*}}\|^2\nonumber\\
&\hspace{-3mm}\leq&\hspace{-1mm}\Delta+\|\nabla h^t(\bm{\mu}^t(k))-\nabla h^t(\bm{\mu}^{t*})\|^2\nonumber\\
&\hspace{-3mm}\leq&\hspace{-1mm}\Delta+{(\sigma^2_g/\sigma_c+\phi)^2} \|\bm{\mu}^t(k)-\bm{\mu}^{t*}\|^2,\nonumber
\end{eqnarray}
where the last inequality is due to the Lipschitz continuity of $\nabla h^t(\bm{\mu}^{t})$. We then obtain the following inequality:
\begin{eqnarray}
&&E\big[\|\bm{\mu}^t(k+1)-\bm{\mu}^{t*}\|^2\big|\bm{\mu}^{t}(k)\big]\nonumber\\
&\leq &(1+\varepsilon^2{(\sigma^2_g/\sigma_c+\phi)^2}-2\varepsilon\sigma_h)\|\bm{\mu}^t(k)-\bm{\mu}^{t*}\|^2+\varepsilon^2\Delta \, .\nonumber
\end{eqnarray}
By taking the total expectation on both sides and by recursively computing  the  steps above, we obtain:
\begin{eqnarray}
&\hspace{-3mm}& \hspace{-1mm}E\big[\|\bm{\mu}^t(k+1)-\bm{\mu}^{t*}\|^2\big]\nonumber\\
&\hspace{-3mm}\leq&\hspace{-1mm} (1+\varepsilon^2{(\sigma^2_g/\sigma_c+\phi)^2}-2\varepsilon\sigma_h)^kE\big[\|\bm{\mu}^t(1)-\bm{\mu}^{t*}\|^2\big]\nonumber\\
&\hspace{-3mm}&\hspace{2mm}+\varepsilon^2\Delta\frac{1-(1+\varepsilon^2{(\sigma^2_g/\sigma_c+\phi)^2}-2\varepsilon\sigma_h)^k}{2\varepsilon\sigma_h-\varepsilon^2{(\sigma^2_g/\sigma_c+\phi)^2}}.\nonumber
\end{eqnarray}
{ Note that \eqref{eq:sigmarelation} always guarantees that the term $(1+\varepsilon^2{(\sigma^2_g/\sigma_c+\phi)^2}-2\varepsilon\sigma_h)$ is nonnegative.} Then with $\varepsilon$ chosen as in \eqref{eq:stepconverge}, the result \eqref{eq:convergeresult} follows.
\end{proof}

Notice that \eqref{eq:stepconverge} allows to select a stepsize $\epsilon>0$ that is ``small enough'' to satisfy the converge requirement. {Meanwhile, the value of \eqref{eq:convergeresult} is decreasing in $\epsilon$, approaching zero if $\epsilon\rightarrow 0$. Therefore, we can choose appropriate $\epsilon$ to satisfy certain desired accuracy.}
It is also worth pointing out  that when $\Delta=0$ (i.e., no devices with discrete power commands are present),  the right-hand-side of \eqref{eq:convergeresult} goes to zero, corresponding to (\ref{eq:dualalg}) reducing to the standard dual gradient algorithm when no discrete variables are present.

\begin{corollary}
When $\cD_{S_i} = \emptyset$ for all $i\in\cN$ (i.e., there are no discrete optimization variables),  and under Assumptions~\ref{as:Cinv}--\ref{ass:sla}, if the  stepsize $\varepsilon$ satisfies  (\ref{eq:stepconverge}), then~(\ref{eq:dualalg}) converges to the exact {saddle point of \eqref{eq:Lag_phi}} asymptotically; i.e.,
\begin{eqnarray}
& \lim_{k\rightarrow\infty}\|\bm{\mu}^t(k)-\bm{\mu}^{t*}\|^2=0 \label{eq:convergeresult2} \\
& \lim_{k\rightarrow\infty}\|\bm{z}^t(k)-\bm{z}^{t*}\|^2=0 \, .
\end{eqnarray}
\end{corollary}


{
\begin{remark}
Notice that the power setpoints from our stochastic dual algorithm may violate the network-level constraint \eqref{syscons} or  the device-level constraint \eqref{eq:statebound0}. In addressing the non-convexity brought in by the discrete devices, we trade possible violation of certain constraints for low computational complexity. In practice, small and/or temporary violation of certain constraints is expected and normal. Moreover, when there is large penetration of volatile renewable generation, it is more appropriate to have statistical guarantee that ensures that the probability of constraint violation is small enough. Although design for statistical performance guarantee is beyond the scope of this paper, for illustrative purpose in the Appendix we have characterized the bound on the variance of the network-level constraint function $g_t(z_t)$ due to the randomized decisions of discrete devices, as well as bounds on probability of voltage deviation, based on which we can choose tighter voltage bounds for the algorithm to ensure the probability of voltage violation with respect to original voltage bounds is below the specified value. The device-level constraints such as temperature constraints of  HVAC can be analyzed in similar way.  
\end{remark}
}

We next present a distributed implementation of the proposed method, along with a receding horizon optimization strategy.

\subsection{Distributed Implementation and Receding Horizon Optimization}\label{sec:distributed}

The stochastic dual algorithm~\eqref{eq:dualalg} can be implemented in a distributed fashion by leveraging the decomposability of the Lagrangian function. Specifically, step~\eqref{eq:random} is decomposable on a per-node basis, where each customer/node $i$ can update the power commands of the devices $\cD_i$ once the vectors  $(\bm{\alpha}^t,\bm{\beta}^t)$ are received. On the other hand, the dual step is performed by a network operator.  The overall distributed algorithm is tabulated as Algorithm~\ref{alg:distributed}.

Problem $(\bm{\mathcal{P}_3^{t}})$ is a multi-period problem. To optimize the operation of both network and devices,  problem $(\bm{\mathcal{P}_3^{t}})$ can be embedded into a receding horizon control (RHC) strategy~\cite{mattingley2011receding} where: 

\noindent (i) at time $t$, the temporal window of $w$ time slots $\cT_{k} := \{t_k, t_{k+1}, \ldots, t_{k+w}\}$ is considered, and the offline Algorithm~\ref{alg:distributed} is utilized to solve $(\bm{\mathcal{P}_3^{t}})$ to convergence; 

\noindent (ii) the solution $\{z_{i,d}^{t *}\}$ corresponding to the first slot $t$ is implemented;  

\noindent (iii) the temporal window is shifted of one time slot $\cT_k \rightarrow \cT_{k+1}$; and,

\noindent (iv) once the window is shifted, point (i) is repeated. 

This strategy is consistent with traditional RHC methods. However, the premises here are that:
 
\noindent $\bullet$ Each time slot is ``long enough'' to allow the offline Algorithm~\ref{alg:distributed} to converge to the solution of  $(\bm{\mathcal{P}_3^{t}})$;
 
\noindent $\bullet$ Within each time slot, the problem inputs are not changing; i.e., prevailing ambient and operational conditions are invariant over a  time slot. 

In the following section, we will present an \emph{online} algorithm that can cope with cases where prevailing ambient and operational conditions vary fast, and they lead to problem inputs that change even within an iteration (or a few iterations) of the Algorithm~\ref{alg:distributed}. The resultant algorithm can be interpreted as a real-time RHC strategy where the temporal window is shifted every iteration of the online algorithm (with the shift being determined by the computational time of the algorithmic steps).

\begin{remark}
Since the paper primarily focuses on the design and analysis of Algorithm~\ref{alg:distributed} and its online counterpart presented in Section~\ref{sec:online}, errors in the forecasting of problem inputs (such as  maximum available PV generation, uncontrollable loads, etc) are not considered in the paper. On the other hand, the proposed framework could be equally applicable to robust or chance-constrained counterparts of $(\bm{\mathcal{P}_3^{t}})$, so long as the  optimization problem is convex. 
\end{remark}

\begin{algorithm}
\caption{Offline Distributed Algorithm} \label{alg:distributed}
At time $t$,
\begin{algorithmic}
\REPEAT
\STATE[S1] Node $i$ measure the state $x^t_i$.
\STATE[S2] Given $(\bm{\alpha}_i^{t}(k),\bm{\beta}_i^{t}(k))$, node $i$ computes $\bm{z}^{t*}(k+1)$ by solving $(\bm{\mathcal{P}^{t}_{1,i}})$, and recovers feasible power set points ${z}^{t}(k+1)$ for $d \in \cD_{S_i}$ via randomization.
\STATE[S3] Node $i$ updates $\bm{p}^{t}_i(k+1)$ and $\bm{q}^{t}_i(k+1)$ based on \eqref{eq:aggregator} and sends the results to the network operator.
\STATE[S4] Network operator calculates system states for $\tau\in W_t^w$ as: 
	\begin{eqnarray}
		\hat{y}^{\tau}(k+1)&\hspace{-2mm}=&\hspace{-2mm} A{p}^{\tau}(k+1)+B{q}^{\tau}(k+1) + c.\label{eq:voltupdate}\nonumber
	\end{eqnarray}
\STATE[S5] Network operator updates dual variables as:
	\begin{eqnarray}
		\bm{\mu}^{t}(k+1)& \hspace{-2mm}=& \hspace{-2mm}\big[\bm{\mu}^{t}(k)+\varepsilon\big( \bm{g}^{t}(\hat{\bm{y}}^{t}(k+1)){-\phi\bm{\mu}^{t}(k)}\big)\big]_+.\nonumber
	\end{eqnarray}
\STATE[S6] Network operator computes the following quantities  for $\tau\in W_t^w$:
	\begin{eqnarray}
		\hspace{-4mm}\alpha^{\tau}(k+1)& \hspace{-2mm}=& \hspace{-2mm}-A^{\top}\nabla^{\top}_{\hat{y}} g^{\tau}(\hat{y}^{\tau}(k+1))\mu^{\tau}(k+1),\nonumber\\
		\hspace{-4mm}\beta^{\tau}(k+1)& \hspace{-2mm}=& \hspace{-2mm}-B^{\top}\nabla^{\top}_{\hat{y}} g^{\tau}(\hat{y}^{\tau}(k+1))\mu^{\tau}(k+1),\nonumber
	\end{eqnarray}
and sends them to the nodes.
\UNTIL meeting stopping criterion {(e.g., 	$\|\bm{\mu}^{t}(k+1)-\bm{\mu}^{t}(k)\|\leq \delta$ for some small $\delta>0$.)} 
\end{algorithmic}
\end{algorithm}

\section{Online and Asynchronous Optimization}\label{sec:online}

In this section, we address the case where computational and communication constraints   may render infeasible the distributed solution of  $(\bm{\mathcal{P}^{t}_3})$---and, hence,  $(\bm{\mathcal{P}^{t}_2})$---to convergence (i.e., offline solution) at a timescale that is consistent with the variability of the underlying ambient and operational conditions.  We consider a situation with: 

(i)~Real-time implementation: power setpoints are implemented whenever they are updated numerically.

(ii)~Asynchronous computation: power setpoints for devices in the set $\cD_{F_i}$ (with continuous decision variables) are updated every iteration of the algorithm (e.g., every second or a few seconds); on the other hand, the setpoints of device $\cD_{S_i}$ (with discrete decision variables) are updated at a slower time scale (depending on, e.g., the availability of the device). This leads to an \emph{asynchronous} control algorithm, where different devices are controlled at different timescales. 

(iii)~The temporal window of the multi-period problem is advanced at each iteration (or every few iterations, depending on particular implementations) of the algorithm. 


(iv)~The resultant online algorithm will leverage \emph{feedback} (see also~\cite{zhou2017incentive2,dallanese2016optimal}), to substitute the linearized model~\eqref{eq:linearization} with measurements of $\hat{y}^t$ for the first timeslot of the optimization window. This accounts for the nonlinear power flows and also helps reduce communication and computation overhead.

\begin{figure}
\centering
\includegraphics[trim={2cm 0 0 0},clip,width=.55\textwidth]{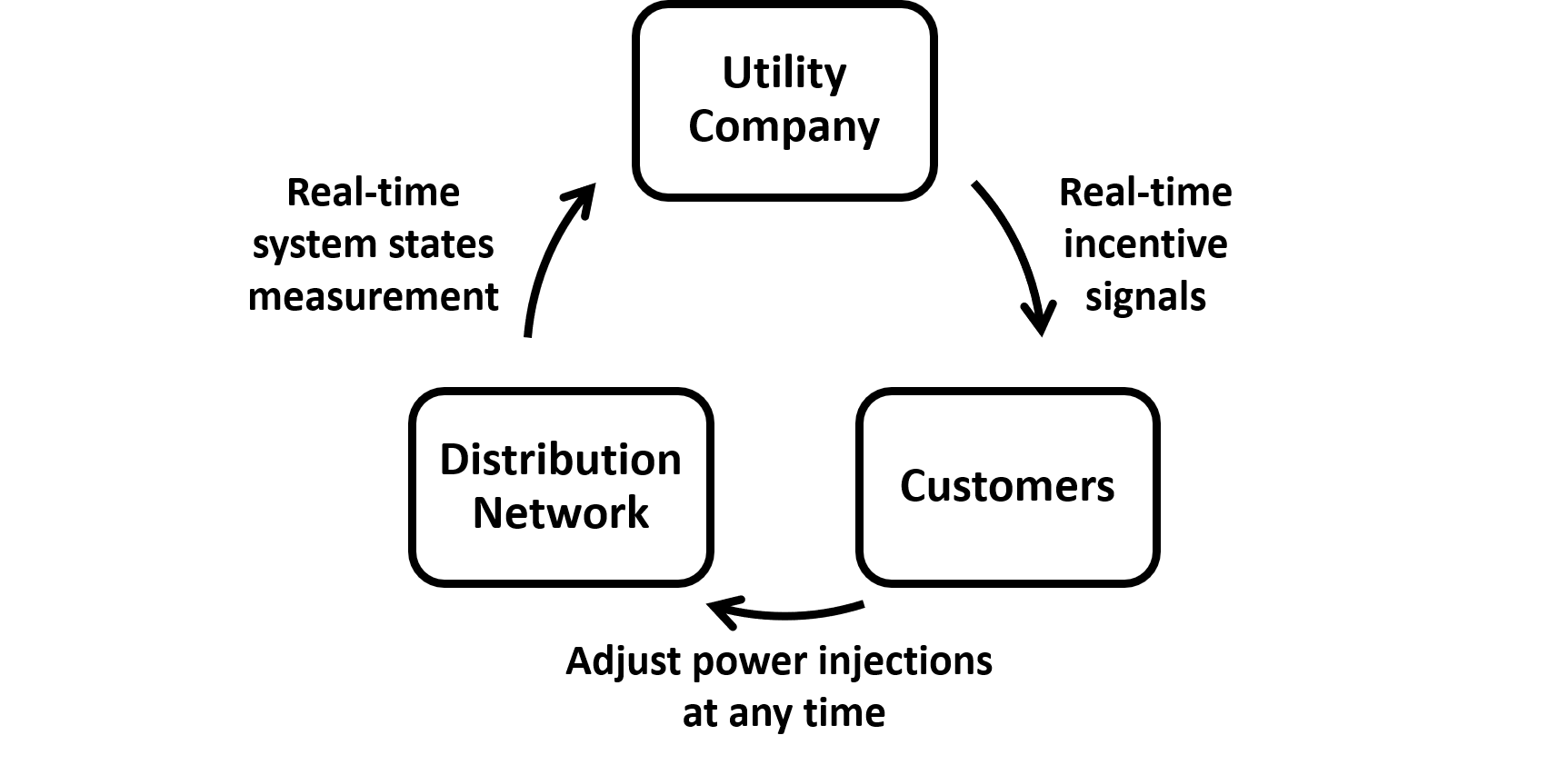}
\caption{Illustration of the online asynchronous algorithm.}
\label{F_topo}
\end{figure}

\subsection{Online Asynchronous Algorithm}

Let $\cT_{i,d}:=\{t^1_{i,d}, t^2_{i,d}, \ldots\}$ denote the set of time indexes where the device $d$ of node $i$ updates the local variables. Let $|\cT_{i,d}|$ denote the cardinality of  $\cT_{i,d}$.
%
Let $\cD_i^t \subset \cD_i$ be the subset of devices that update power setpoints at time $t$, and let $\cD^t=\bigcup_{i\in\cN}\cD^t_i$.
The quantities $\{\bm{z}^t_{i,d}\}_{d\not\in\cD^t}$ are treated as constants at time $t$; thus, define a ``reduced" Lagrangian as
\begin{eqnarray}
\cL_r^t\big(\{\bm{z}^t_{i,d}\}_{d\in\cD^t},\bm{\mu}_r^t\big):=\cL^t\big(\{\bm{z}^t_{i,d}\}_{d\in\cD^t},\bm{\mu}^t\big|\{\bm{z}^t_{i,d}\}_{d\not\in\cD^t}\big),
\end{eqnarray}
and define its corresponding ``reduced" dual function as:
\begin{eqnarray}
h_r^t(\bm{\mu}_r^t):=\min_{\{\bm{z}^t_{i,d}\}_{d\in\cD^t}} \cL_r^t\big(\{\bm{z}^t_{i,d}\}_{d\in\cD^t},\bm{\mu}_r^t\big),
\end{eqnarray}
whose unique  optimal solution is: 
\begin{eqnarray}
\bm{\mu}_r^{t*}:= \underset{\bm{\mu}_r^t\geq 0}{\arg\max} ~h_r^t(\bm{\mu}_r^t).
\end{eqnarray}
Based on the definitions above, 
we propose the online asynchronous algorithm tabulated as Algorithm~\ref{alg:ondistributed}. Notice that we have included an additional subscript to denote the time when decision variables are calculated; e.g., $\bm{\mu}^{t+1}_{r,t}$ is the dual vector for the reduced Lagrangian from $t+1$ to $t+w+1$ calculated at time $t$. We also recall  that in Algorithm~\ref{alg:ondistributed}, only \emph{one} iteration is carried out at time $t$, and then the temporal window is shifted from $t$ to $t + 1$.


\begin{algorithm}
\caption{Online Asynchronous  Algorithm} \label{alg:ondistributed}
\begin{algorithmic}
\STATE At time $t$,
\STATE[S1] Node $i$ updates measures $x_{i,t}^t$.
\STATE[S2] Given $(\bm{\alpha}_{i,t-1}^{t},\bm{\beta}_{i,t-1}^{t})$, Node $i$ sets $\bm{z}^t_{i,d,t}=\bm{z}^{t-1}_{i,d,t-1},\ \forall d\not\in\cD_i^t$, and solves $(\bm{\mathcal{P}^{t}_{1,i}})$ over $\{\bm{z}^t_{i,d,t}\}_{d\in\cD_i^t}$ to get the solution $\{\bm{z}^{t*}_{i,d,t}\}_{d\in\cD_i^t}$.   Recover and implement discrete power set points for $\cD_{S_i}$ via randomization, with $E[{z}_{i,d,t}^{t}]={z}_{i,d,t}^{t*}$.
\STATE[S3] Node $i$ calculates future aggregated power $(z^{t+1*}_{i,t},\ldots, z^{t+w*}_{i,t})$ based on \eqref{eq:aggregator}, and sends the result to network operator.
\STATE[S4] Network operator measures the current states $y_t^{t}$, and estimates the future states $\hat{y}_t^{\tau}$ for $\tau\in W_t^w$ by 
	\begin{eqnarray}
	\hat{y}_t^{\tau}&\hspace{-2mm}=&\hspace{-2mm} A{p}_t^{\tau*}+B{q}_t^{\tau*}+ c.\nonumber
	\end{eqnarray}
\STATE[S5] Network operator updates the dual variables for $t+1$ with \emph{measured} states by	
	\begin{eqnarray}
		\mu^{t+1}_{r,t}& \hspace{-2mm}=& \hspace{-2mm}\big[\mu^{t}_{r,t-1}+\varepsilon\big( g^{t}({y}_t^{t}){-\phi	\mu^{t}_{r,t-1}}\big)\big]_+.\nonumber
	\end{eqnarray}
and updates the dual variables for $\tau=t+1,\ldots,t+w$ with \emph{predicted} states by
	\begin{eqnarray}
		\mu^{\tau+1}_{r,t}& \hspace{-2mm}=& \hspace{-2mm}\big[\mu^{\tau}_{r,t-1}+\varepsilon\big( g^{\tau}(\hat{y}_t^{\tau}){-\phi	\mu^{\tau}_{r,t-1}}\big)\big]_+.\nonumber
	\end{eqnarray}
\STATE[S6] Network operator updates signals for $t+1$ by
	\begin{eqnarray}
		\alpha^{t+1}_t=-A^{\top}\nabla^{\top}_{\hat{y}} g^{t}({y}_t^{t})\mu^{t+1}_{r,t},\  \beta^{t+1}_t=-B^{\top}\nabla^{\top}_{\hat{y}} g^{t}({y}_t^{t})\mu^{t+1}_{r,t},\nonumber\hspace{-10mm}
	\end{eqnarray}		
and those for $\tau=t+1,\ldots,t+w$ by
	\begin{eqnarray}
		\alpha^{\tau+1}_t\!=\!-A^{\top}\nabla^{\top}_{\hat{y}} g^{\tau}(\hat{y}_t^{\tau})\mu^{\tau+1}_{r,t}\!,\  \beta^{\tau+1}_t\!=\!-B^{\top}\nabla^{\top}_{\hat{y}} g^{\tau}(\hat{y}_t^{\tau})\mu^{\tau+1}_{r,t}\!,\!\nonumber\hspace{-10mm}
	\end{eqnarray}						
and sends the results $(\bm{\alpha}_{i,t}^{t+1},\bm{\beta}_{i,t}^{t+1})$ to node $i$.
\STATE[S7] Shift temporal window from $t$ to $t + 1$, and go to [S1]. 
\end{algorithmic}
\end{algorithm}

\subsection{Performance Analysis}\label{sec:constant}

To analyze the performance of Algorithm~\ref{alg:ondistributed}, the following assumptions are made. 
\begin{assumption}
There exists some constant $e>0$ such that the variation of the optimal $\bm{\mu}_r^*$ of $(\bm{\mathcal{P}^{t}_3})$ over any two consecutive timeslots is bounded as
\begin{eqnarray}
 \|\bm{\mu}_r^{t+1*}-\bm{\mu}_r^{t*}\|^2 \leq e,\ \forall t.\label{eq:error}
\end{eqnarray}
\end{assumption}
This is a standard assumption in the domain of time-varying optimization~\cite{SimonettoGlobalsip2014, dallanese2016optimal} to characterize the variability of optimal solutions from the current timeslot to the next (and, hence, the variability of problem inputs).

Let $\hat{\bm{y}}^t:=[\hat{y}^{t\top},\hat{y}^{t+1\top},\ldots,\hat{y}^{t+w\top}]^{\top}$, and ${\bm{y}}^t:=[{y}^{t\top},\hat{y}^{t+1\top},\ldots,\hat{y}^{t+w\top}]^{\top}$, where the first element of the vector $\hat{\bm{y}}^t$ is  measured from the grid. We next assume bounded error due to the linearization of power flow equations.

\begin{assumption}\label{ass:rho}
Given any feasible power injection $\bm{z}^t$, there exists a constant $\rho>0$ such that
\begin{eqnarray}
\|\bm{g}^t(\bm{y}^t(\bm{z}^t))-\bm{g}^t(\hat{\bm{y}}^t(\bm{z}^t))\|^2 \leq \rho.\label{eq:rho}
\end{eqnarray}
\end{assumption}

%

{
\begin{remark}
Notice that $\bm{g}^t$ is a continuous function. Assumption~\ref{ass:rho}  is equivalent to assuming a bounded discrepancy  between the nonlinear power flow model and its linear approximation \eqref{eq:linearization}, which always holds given  bounded power injections $\bm{z}^t$  and the fact that both $\bm{y}^t$ and $\hat{\bm{y}}^t$ are continuous functions of $\bm{z}^t$ with bounded derivatives (see, e.g., \cite{zhou2016game}). 
	Moreover, as the linearized power flow model \eqref{eq:linearization} is an accurate
	approximation under normal operating condition, the bound $\rho$ is expected to be small; see, e.g., \cite{baran1989network,christakou2013efficient,linModels}.
\end{remark}
}

Then, the following convergence result can be stated. 

\begin{theorem}\label{the:onlineperform}
Under Assumptions~\ref{as:Cinv}--\ref{ass:rho}, and given a stepsize $\varepsilon$ chosen to satisfy \eqref{eq:stepconverge}, Algorithm~\ref{alg:ondistributed} converges as
\begin{eqnarray}
\limsup_{t\rightarrow\infty}E\big[\| \bm{\mu}_{r,t}^{t+1}-\bm{\mu}_r^{t+1*} \|^2\big]=\frac{\varepsilon^2\Delta+\varepsilon^2\rho+e}{2\varepsilon\sigma_h-\varepsilon^2 {(\sigma^2_g/\sigma_c+\phi)^2}}.\label{eq:onlineconverge}
\end{eqnarray}
\end{theorem}
\begin{proof}
Denote by $\bm{\mu}_{t,r}^{t+1}$ the dual produced by  Algorithm~\ref{alg:ondistributed} at time $t$ and $\bm{\mu}_r^{t+1*}$ the optimal dual solution. We have
\begin{eqnarray}
&&\hspace{-3mm} E\big[\| \bm{\mu}_{t,r}^{t+1}-\bm{\mu}_r^{t+1*} \|^2\big] \nonumber\\
&\hspace{-3mm}\leq&\hspace{-3mm}E\big[\| \bm{\mu}_{r,t-1}^{t}+\varepsilon \bm{g}^{t}(\hat{\bm{y}}_{t-1}^t)+\varepsilon \bm{g}^{t}({\bm{y}}_{t-1}^t)-\varepsilon \bm{g}^{t}(\hat{\bm{y}}_{t-1}^t) \nonumber\\
&\hspace{-3mm}&\hspace{0mm}{-\phi\bm{\mu}_{r,t-1}^{t}}-\bm{\mu}^{t*}_r+\bm{\mu}^{t*}_r-\bm{\mu}_r^{t+1*}\|^2\big]\label{eq:the3step1}\\
&\hspace{-3mm}\leq &\hspace{-3mm}E\big[\| \bm{\mu}_{r,t-1}^{t}\!+\!\varepsilon \bm{g}^{t}(\hat{\bm{y}}_{t-1}^t){-\phi\bm{\mu}_{r,t-1}^{t}}-\bm{\mu}_r^{t*}\|^2\!+\!\|\bm{\mu}_r^{t*}\!-\!\bm{\mu}_r^{t+1*}\|^2 \nonumber\\
&\hspace{-3mm}&\hspace{0mm}+\|\varepsilon \bm{g}^{t}({\bm{y}}_{t-1}^t(\bm{z}_{t-1}^t))-\varepsilon \bm{g}^{t}(\hat{\bm{y}}_{t-1}^t(\bm{z}_{t-1}^t))\|^2\big]\nonumber\\
&\hspace{-3mm}\leq &\hspace{-3mm}E\big[\| \bm{\mu}_{r,t-1}^{t}+\varepsilon \bm{g}^{t}(\bm{z}_{t-1}^t){-\phi\bm{\mu}_{r,t-1}^{t}}-\bm{\mu}_r^{t*}-\varepsilon \bm{g}^{t}(\bm{z}^*)\nonumber\\
&\hspace{-3mm} &{+\phi\bm{\mu}_{r}^{t*}} \|^2\big]+e+\varepsilon^2\rho, \nonumber
\end{eqnarray}
where the last inequality is due to (\ref{eq:error})--(\ref{eq:rho}).
{Similar to the techniques in the proof of Theorem~\ref{the:converge},} by using Lemma~\ref{lem:gradient} and the strong concavity of the dual function (which still hold for the reduced dual function), and recalling that $\Delta$ is the upper bound on the variance $Var(\bm{g}^t(\bm{z}^t))$, we get: 
\begin{eqnarray}
&\hspace{-3mm}& \hspace{-1mm} E\big[\| \bm{\mu}_{r,t}^{t+1}-\bm{\mu}_r^{t+1*} \|^2\big]\nonumber\\
&\hspace{-3mm}\leq &\hspace{-1mm}(1+\varepsilon^2{(\sigma^2_g/\sigma_c+\phi)^2}-2\varepsilon\sigma_h)E\big[\|\bm{\mu}^t_{r,t-1}-\bm{\mu}_r^{t*}\|^2\big]\nonumber\\
&&\hspace{3mm}+\varepsilon^2\Delta+\varepsilon^2\rho+e\label{eq:the3step2}\\
&\hspace{-3mm}\leq &\hspace{-1mm}(1+\varepsilon^2{(\sigma^2_g/\sigma_c+\phi)^2}-2\varepsilon\sigma_h)^tE\big[\|\bm{\mu}^{1}_{r,0}-\bm{\mu}_r^{1*}\|^2\big]\nonumber\\
&&\hspace{3mm}+(\varepsilon^2\Delta+\varepsilon^2\rho+e)\frac{1+(1+\varepsilon^2{(\sigma^2_g/\sigma_c+\phi)^2}-2\varepsilon\sigma_h)^t}{2\varepsilon\sigma_h-\varepsilon^2{(\sigma^2_g/\sigma_c+\phi)^2}},\nonumber
\end{eqnarray}
{where the last inequality is obtained by recursively applying the steps from \eqref{eq:the3step1}--\eqref{eq:the3step2}.}
%

Selecting the stepsize $\varepsilon$ in a way to satisfy \eqref{eq:stepconverge}, and letting $t\rightarrow\infty$, the main result \eqref{eq:onlineconverge} follows.
\end{proof}

The bound~\eqref{eq:onlineconverge} provides a characterization of the discrepancy between the optimal dual variable and the dual variable generated by the online algorithm. The asymptotic bound depends on the underlying dynamics of the optimization problem through $e$ and on the measurement errors and linearization errors through $\rho$. The result~\eqref{eq:onlineconverge}  can also be interpreted as input-to-state stability, where the  trajectory  of the optimal dual variables is taken as a reference.

The result \eqref{eq:onlineconverge} also suggests ways to improve the performance of Algorithm~\ref{alg:ondistributed} in terms of tracking. For example, apply reasonably small stepsize $\epsilon$; more frequent updates leads to smaller $e$; finer control granularity for discrete devices leads to a smaller $\Delta$.


%


\section{Numerical Examples}\label{sec:simulation}


\subsection{Simulation Setup}
%
{ We consider the IEEE 37-node test feeder shown in Fig.~\ref{F_feeder}. We modify the network by considering the phase ``c'' of the original system and replacing the loads specified in the original dataset with the real load data measured from feeders in Anatolia, California during a week of August 2012~\cite{Bank13}. 
The data have a granularity of 1 second, and represent the loading of secondary transformers.
Line impedances, shunt admittances as well as active and reactive loads are adopted from the respective data set. We install one PV systems at each of the following 18 nodes: $4$, $7$, $10$, $13$, $17$, $20$, $22$, $23$, $26$, $28$, $29$, $30$, $31$, $32$, $33$, $34$, $35$ and $36$, with their generation profiles simulated based on the real solar irradiance data available in~\cite{Bank13}. 
The ratings of these inverters are $100$ kVA for $i = 3$, $350$ kVA for $i = 33, 34$,  $300$ kVA for $i = 35, 36$, and $200$ kVA for the remaining inverters. For illustrative purpose, we show in Fig.~\ref{F_VPP_paper_load} the loads realization as well as the available real power from a PV system, where large and fast fluctuation can be observed.}
We then install 15 A/Cs and 15 batteries at each of the following 25 nodes: 2, 5, 6, 7, 9, 10, 11, 13, 14, 16, 18, 19, 20, 21, 22, 24, 26, 27, 28, 29, 30, 32, 33, 35 and 36, totaling 375 A/Cs and 375 batteries.  The detailed simulation modeling of PV inverters, A/Cs, and batteries are described as follows.

\emph{PV inverters}: The PV inverters' objective functions are set uniformly to 
\begin{eqnarray}C_{i,d}^t(p_{i,d}^t,q_{i,d}^t) =   c_p (p^{\textrm{av},t}_{i,d} - p^t_{i,d})^2+c_q q_{i,d}^{t2},\label{eq:PVmodel}
\end{eqnarray} 
with positive constant $c_p$ and $c_q$, in an effort to minimize the amount of real power curtailed from the available power $p^{\textrm{av},t}_{i,d}$ based on irradiance conditions at time $t$, and the amount of reactive power injected or absorbed. PV inverters are set to be updated every 1 second within a convex feasible set as (\ref{eq:pv}).

\emph{A/Cs}: We set a uniform cost function for all A/Cs as 
\begin{eqnarray}
C_{i,d}(\bm{T}_{i,d}^{t})=c_t\|\bm{T}_{i,d}^{t}-\bm{1}\cdot T_{i,d}^{\text{nom}}\|^2, \label{eq:ACmodel}
\end{eqnarray}
where $c_t$ is positive constant, $T_{i,d}^{\text{nom}}$ is a preferred room temperture set at  $75^{\circ}$F, and the future room temperature and constraints are modeled according to (\ref{eq:tp1})--(\ref{eq:tp3})
with $\zeta_1$ uniformly set to $0.1$, and $\zeta_2$ randomly picked from $[-0.0009,-0.0011]$.
The acceptable ranges of room temperatures are set to $[70^{\circ}\text{F},80^{\circ}\text{F}]$. Each A/C updates every 15 minutes with two possible power status: 4~kW (on) and 0 (off).

\emph{Batteries}: We set a uniform cost function for batteries as 
\begin{eqnarray}
C_{i,d}(\bm{S}_{i,d}^{t})=c_b\|\bm{S}_{i,d}^{t}-\bm{1}\cdot S_{i,d}^{\text{nom}}\|^2,\label{eq:EVmodel}
\end{eqnarray}
where $c_b$ is a positive constant, $S_{i,d}^{\text{nom}}$ is a preferred battery state of charge set to $0.5$ for all batteries for illustrative purpose, and the battery dynamics and constraints are modeled as (\ref{eq:e1})--(\ref{eq:e2}), 
with the { (normalized)} SOC bounds set to $[0.2,0.8]$ uniformly for simplicity.\footnote{Here, this battery model is simplified for better illustrating battery's response towards signals. In practice, we can apply objective function to avoid frequent switch between charging and discharging, deep charge/discharge that can damage battery, etc., and the constraints can involve charging deadline, e.g., charge the EV battery with at least 90\% SOC by 8 am.} For each battery, we set {a uniform capacity of 20~kWh corresponding to a (normalized) SOC of 1}, a uniform charging rate of 4~kW, and a fixed discharging rate randomly picked between -3.6~kW and -4.4~kW. { Therefore, charging a battery for 15~min increases its SOC by 0.05 while discharging it for 15~min decreases its SOC by some value between 0.045 and 0.055 depending on the randomly selected discharging rate.} {Each} battery updates power status of charging/off/discharging every 15 minutes.
\begin{figure}
\centering
\includegraphics[width=.43\textwidth]{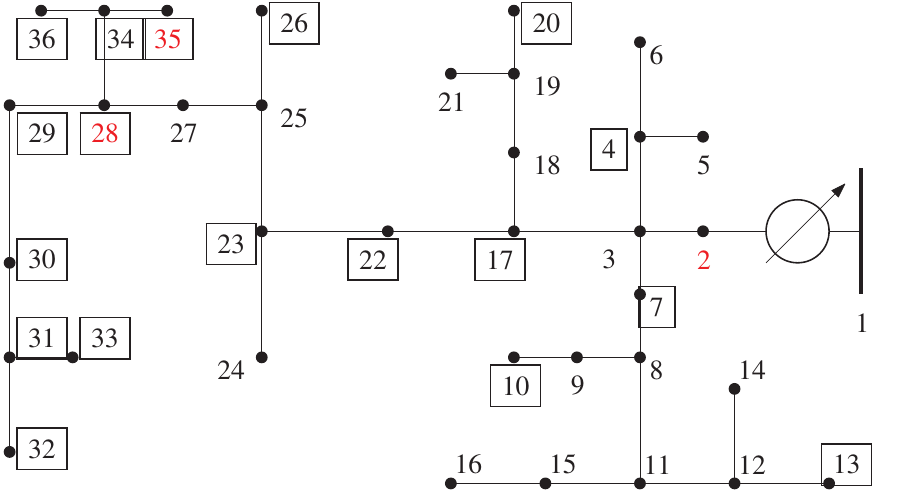}
\caption{Modified IEEE 37-node feeder. The boxes represent PV systems.}
\label{F_feeder}
\end{figure}
\begin{figure}
\centering
\includegraphics[trim = 12mm 0mm 0mm 0mm, clip, scale=0.37]{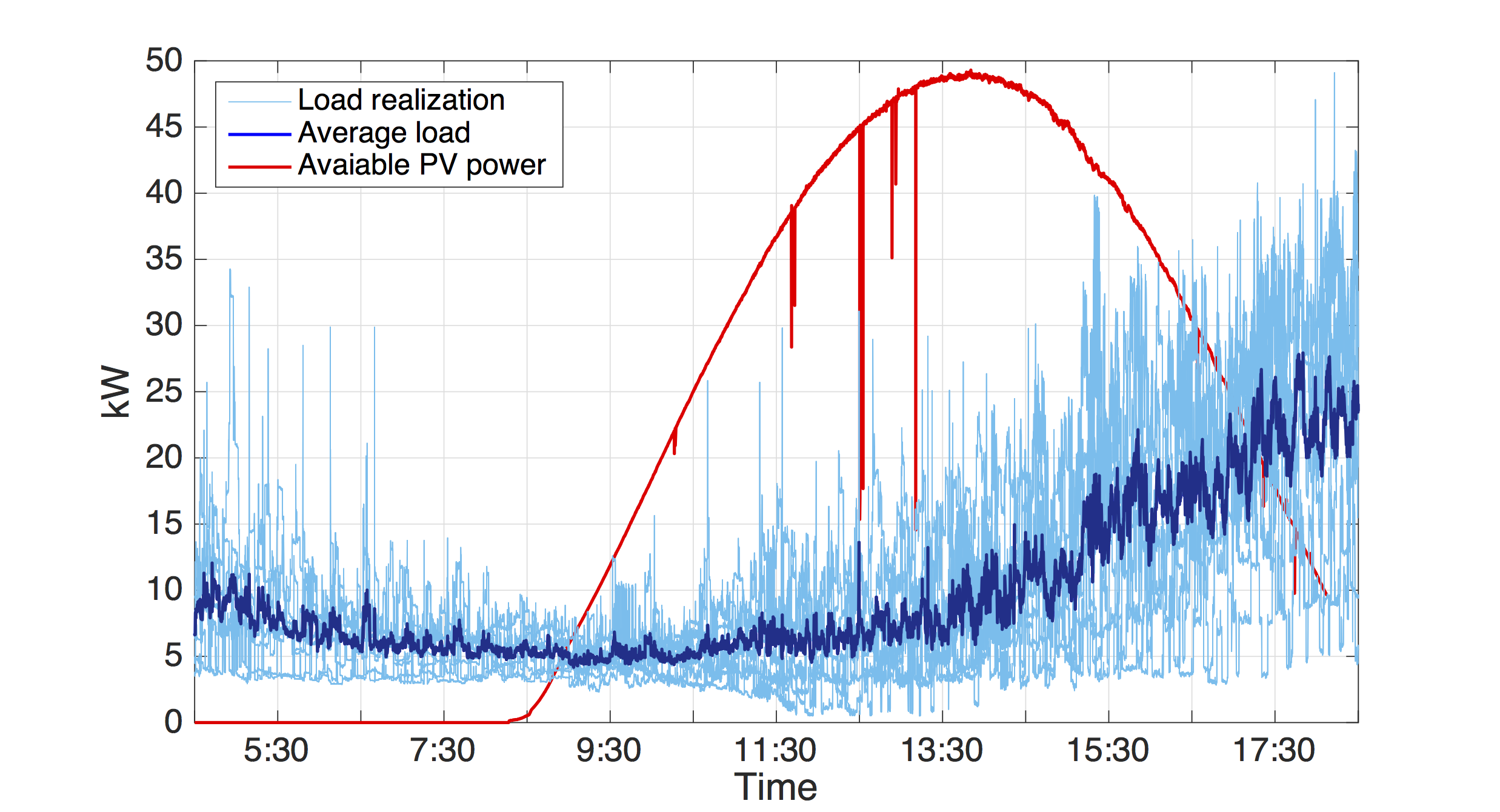}
\caption{Profiles of loads and power available from the PV systems. The average load profile is marked in blue.}
\label{F_VPP_paper_load}
\end{figure}

\emph{Network Operator}: The operational constraints $g^t(\hat{y}^t)\leq 0$ are set to voltage regulation:
\begin{eqnarray}
\underline{v}_i \leq \hat{v}_i^t \leq \overline{v}_i,\label{eq:voltregu}
\end{eqnarray} 
{with voltage upper and lower bounds $\overline{v}_i$ and $\underline{v}_i$  set to $1.05$~p.u. and $0.95$~p.u. respectively for all $i\in\hN$. We further enforce robust design by implementing tighter voltage bounds of $\overline{v}'_i=1.04$~p.u. and $\underline{v}'_i=0.96$~p.u.; see Appendix~\ref{app:robustvolt} for detailed discussion for robust voltage regulation design. Network operator collects voltage magnitudes from all nodes, updates dual variables with a constant stepsize, 
and calculates incentive signals every second.}  { Notice that we have used very small regularization parameter $\phi>0$ as well as $\phi=0$, and there is no distinguishable difference between the two.  All the results reported below are for the case with  $\phi=0$.}

{The time window is one hour ahead. At any time $t$, we have $W^{w}_t=\{t, t+15~\text{minutes}, t+30~\text{minutes}, t+45~\text{minutes}\}$. The randomization step is performed based on \eqref{eq:probability}, 
since the feasible operational points of both A/Cs and batteries are in one dimension.}

For the rest of this section, we will focus on illustrating performance of online asynchronous algorithm, i.e., Algorithm~\ref{alg:ondistributed}. We refer to numerical results in \cite{zhou2017stochastic} for numerical examples of stability analysis for offline algorithm. 

\subsection{Online Asynchronous Distributed Algorithm}

We implement Algorithm~\ref{alg:ondistributed} to coordinate network operator and customers to achieve operational and economic goals,\footnote{{ While there are existing solvers for convex optimization problems such as YALMIP, CPLEX, and GAMS  \cite{cplex2009v12, lofberg2004yalmip, sahinidis2004baron}, in simulations we implement the standard primal-dual gradient algorithm to  solve  problem $\bm{\mathcal{P}}_{1,i}^t$ in step [S2] of Algorithm~\ref{alg:ondistributed}. }} and examine individual and aggregate behaviors of the controllable devices. The asynchrony is simplified as follows: At the beginning of each minute, one-{fifteenth (i.e., 25/375)} of the A/Cs and batteries update their setpoints while the rest maintain the previous decisions. {Therefore, every 15 minutes, all these devices complete one round of updates.} In practice, more asynchronous updates are expected, which usually leads to better and smoother results. 
{ The simulation is for a period of time from~4:30 am to 6:15~pm (49500 seconds), and conducted on a laptop with Intel Core i7-7600U CPU @ 2.80GHz 2.90 GHz, 8.00GB RAM, running MatLab R2018a on Windows 10 Enterprise Version. The update frequency for the continuous devices is 1 iteration per second. The total computing time and the average per-iteration computing time are 1150 seconds ($<20$ minutes) and about 0.02 second, respectively.}
\begin{figure}
\centering
\includegraphics[trim = 0mm 0mm 0mm 0mm, clip, scale=0.32]{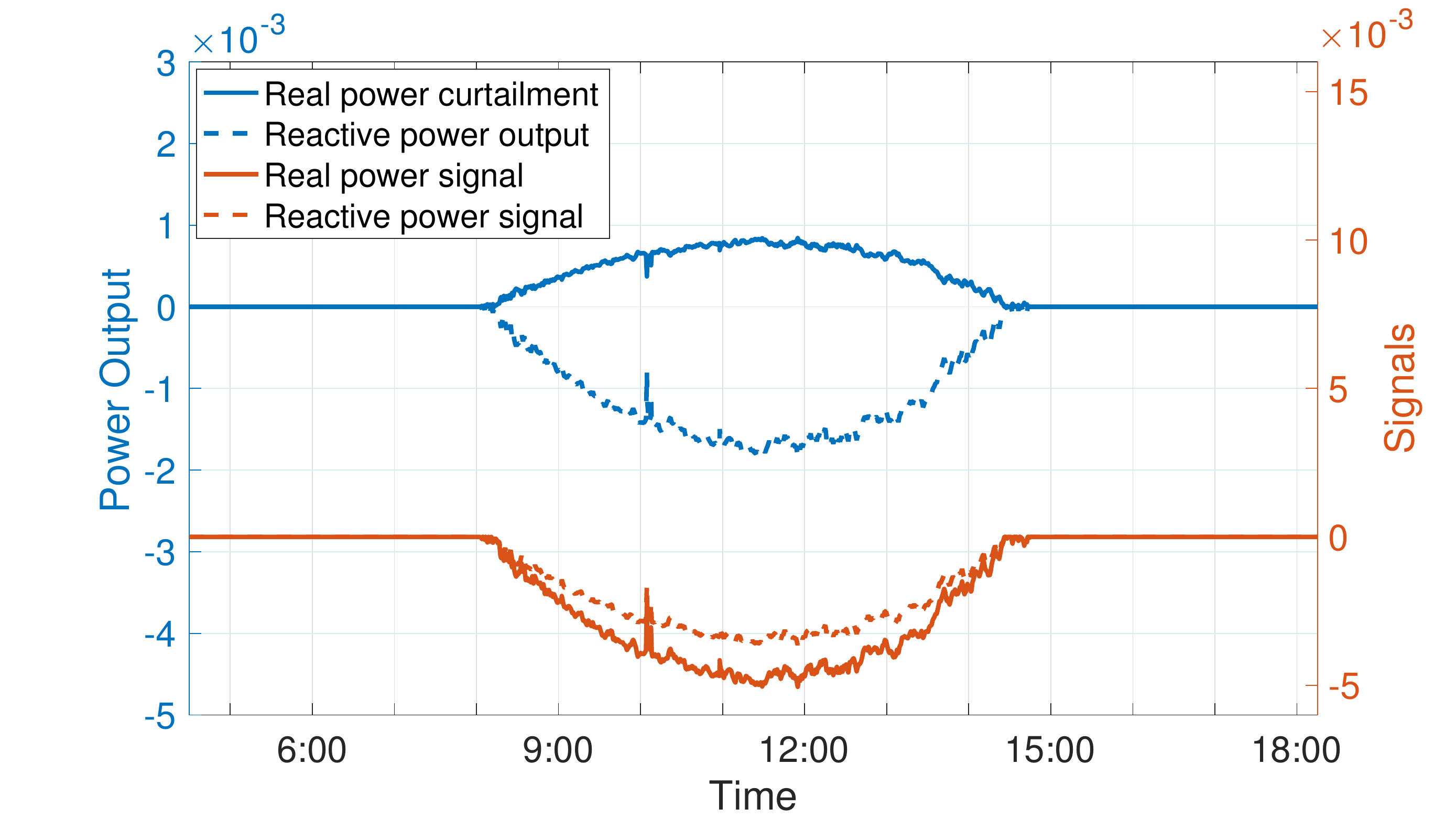}
\caption{Power output of one arbitrary PV inverter.}\label{fig:1pv}
\vspace{5mm}
\includegraphics[trim = 0mm 0mm 0mm 0mm, clip, scale=0.32]{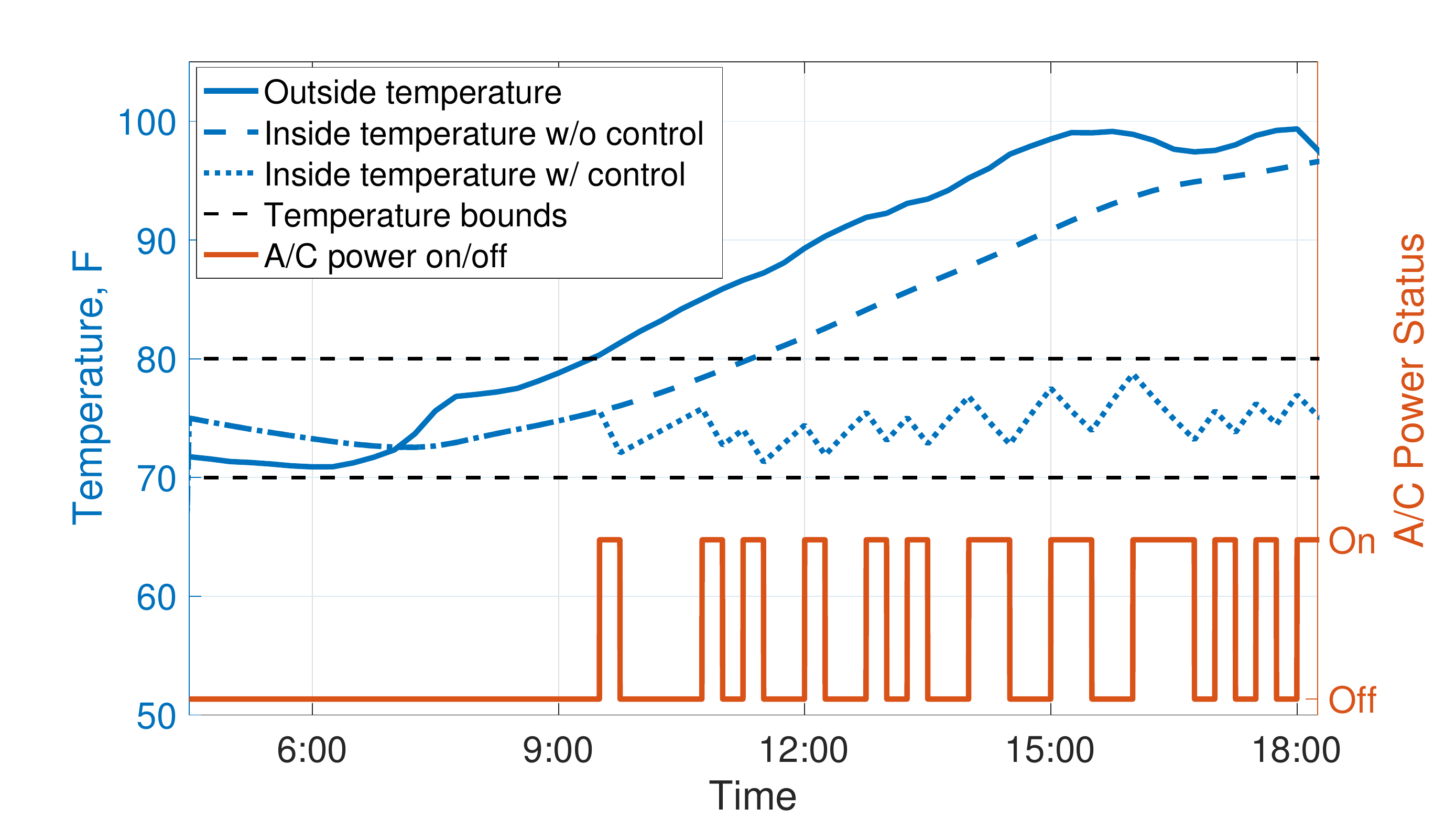}
\caption{Temperature and power status of one arbitrary A/C.}\label{fig:1ac}
\vspace{5mm}
\includegraphics[trim = 0mm 0mm 0mm 0mm, clip, scale=0.425]{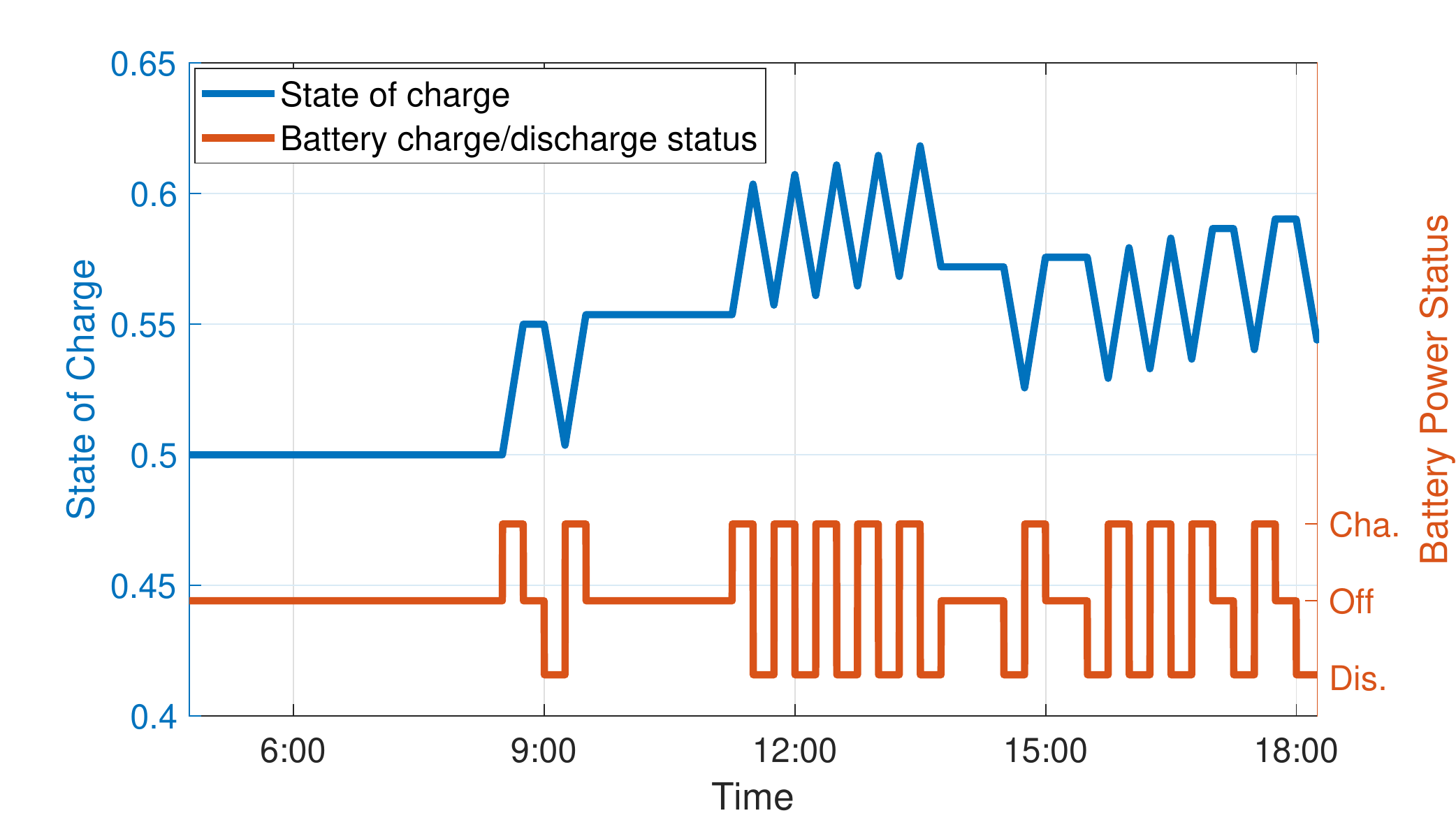}
\caption{SOC and power status of one arbitrary battery.}\label{fig:1soc}
\end{figure}

\subsubsection{Dynamics of Individual Device}
We first zoom in to examine individual controllable devices.

\emph{PV inverter:} We select an arbitrary PV inverter and plot its power output in Fig.~\ref{fig:1pv}. Positive real power curtailment and negative reactive power injection can be observed in response to the signal that incentivizes negative power injection.

\emph{A/C:} We select an arbitrary house from an arbitrary node and plot its inside temperature. As shown in Fig.~\ref{fig:1ac}, the room temperature is controlled within the acceptable range. 

\emph{Battery:} We select an arbitrary battery and plot its SOC along with its power setpoint in Fig.~\ref{fig:1soc}.  {The selected battery has a discharging rate of -3.85~kW. Here, battery power status ``Cha.", ``Off", and ``Dis." displayed on the right y-axis means charging, off, and discharging with 4~kW, 0~kW and -3.85~kW, respectively.} We observe an SOC near 0.5 but deviating due to the response to incentive signal, which we will further illustrate later.

\subsubsection{Aggregated Behavior}
We zoom out to examine the aggregated behavior of hundreds of discrete devices.

As shown in Fig.~\ref{fig:375ac}, the temperatures of all 375 houses with A/Cs are controlled within the acceptable temperature range, with an average temperature around the preferred value of $75^{\circ}$F. Tighter temperature bounds can be achieved by finer control granularity and more frequent control.  { It is also expected that larger window size may provide better temperature control by considering future temperatures of a longer time horizon, as shown in Fig.~\ref{fig:375ac_w}. }

\begin{figure}
\centering
\includegraphics[trim = 0mm 0mm 0mm 0mm, clip, scale=0.44]{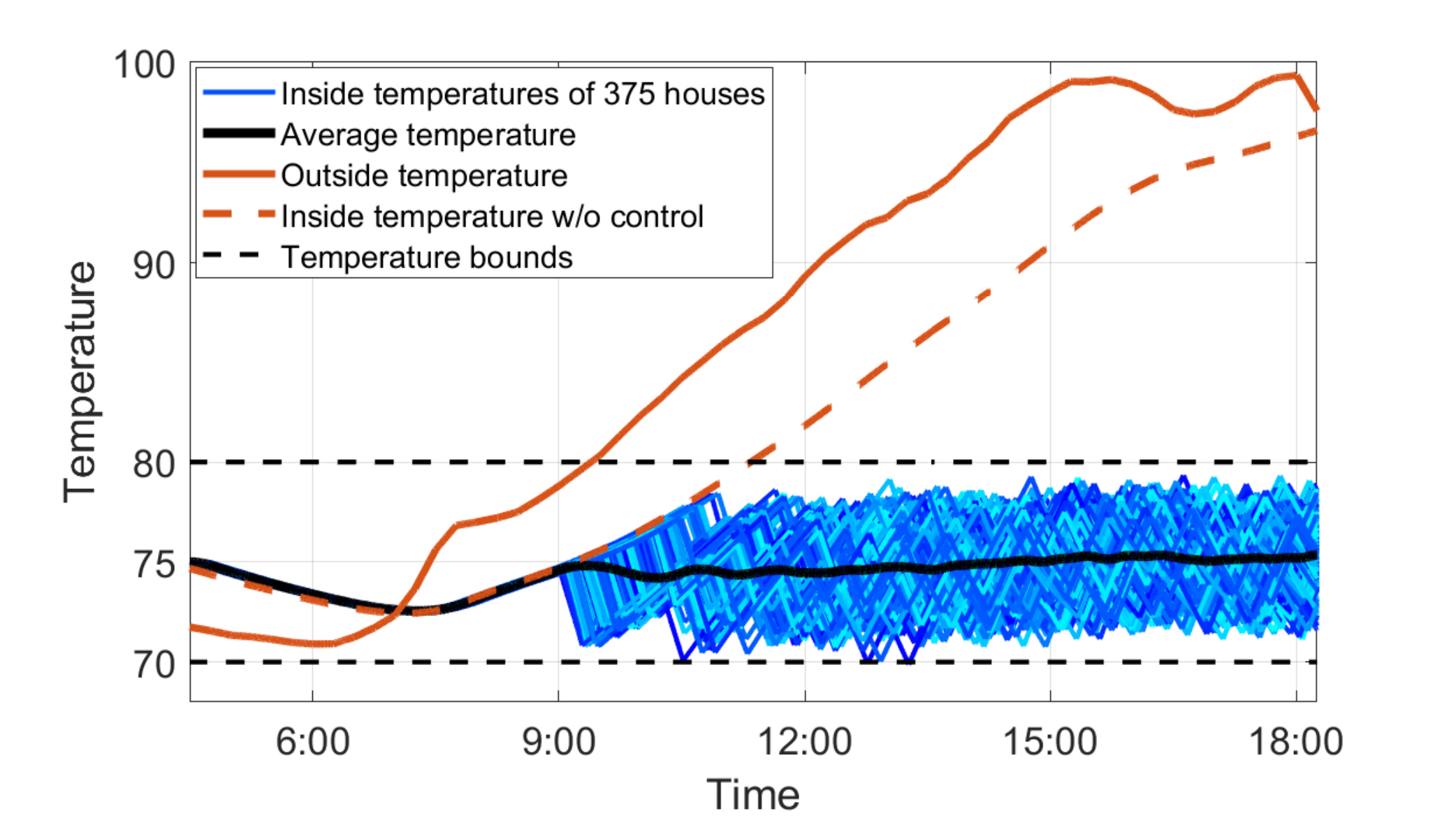}
\caption{Temperatures of 375 households under control by A/Cs.}\label{fig:375ac}
\vspace{5mm}
\includegraphics[trim = 5mm 0mm 0mm 0mm, clip, scale=0.33]{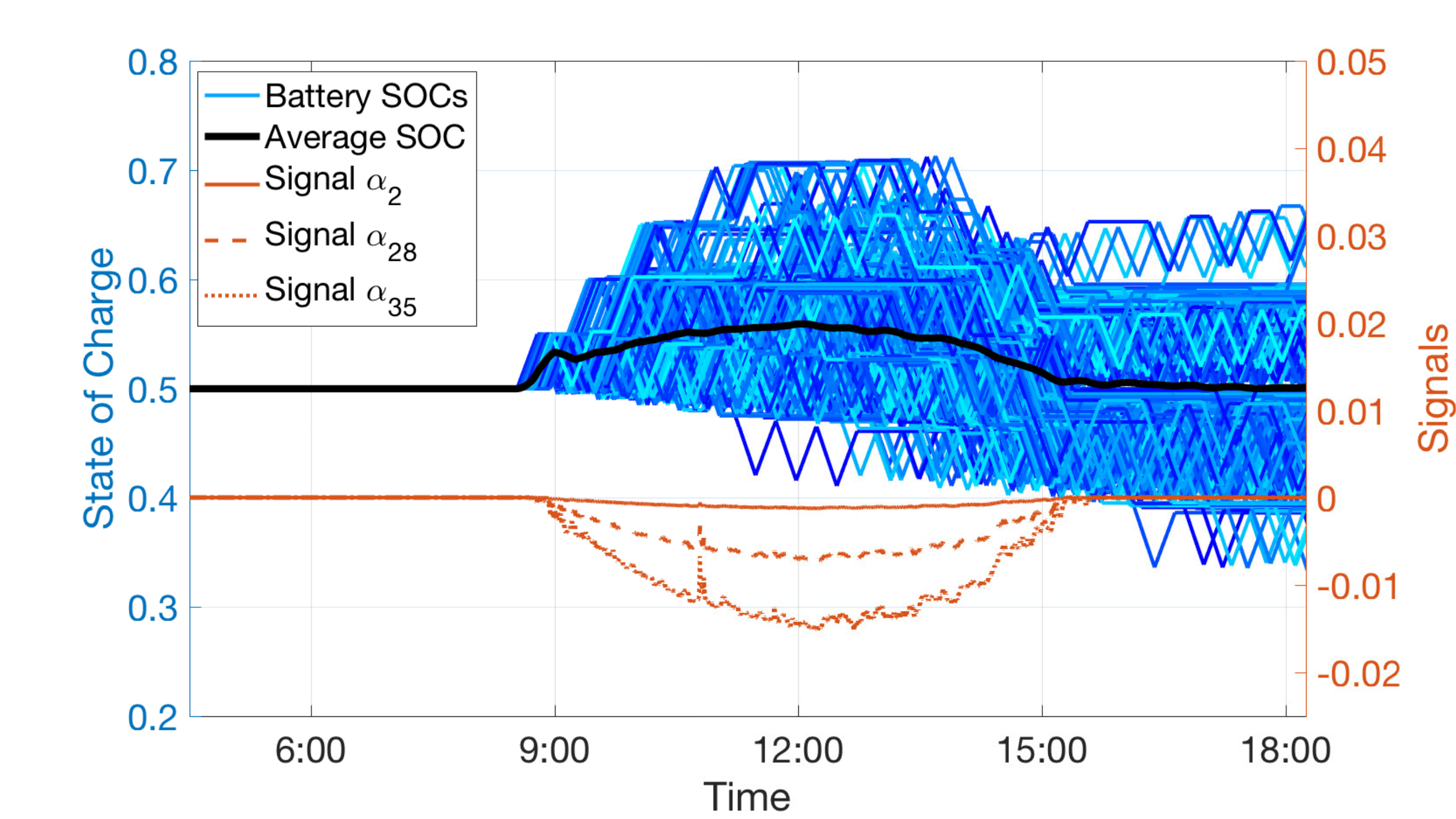}
\caption{SOCs of 375 batteries under control and signals.}\label{fig:375b}
\vspace{5mm}
\includegraphics[trim = 0mm 0mm 0mm 0mm, clip, scale=0.33]{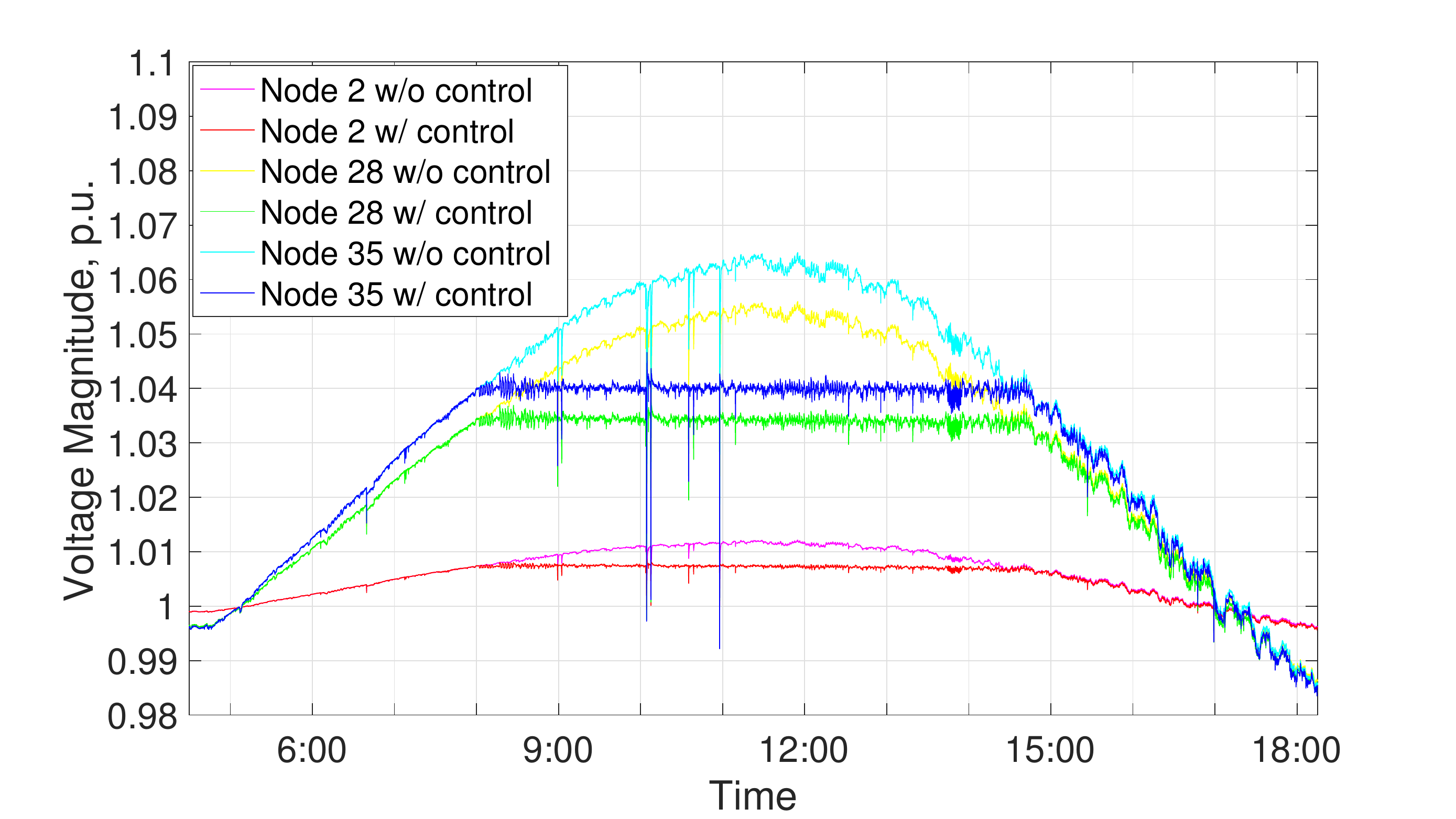}
\caption{Uncontrolled and controlled voltages.}\label{fig:controlv}
\end{figure}

Fig.~\ref{fig:375b} shows the SOCs of all 375 batteries, together with exemplifying incentive signals for real power at 3 of nodes (see Eq.~\eqref{eq:signal}). One can observe that batteries are incentivized to charge more during the middle of the day with average SOC higher than the preferred one.

\begin{figure}
\centering
\includegraphics[trim = 23mm 0mm 0mm 0mm, clip, scale=0.365]{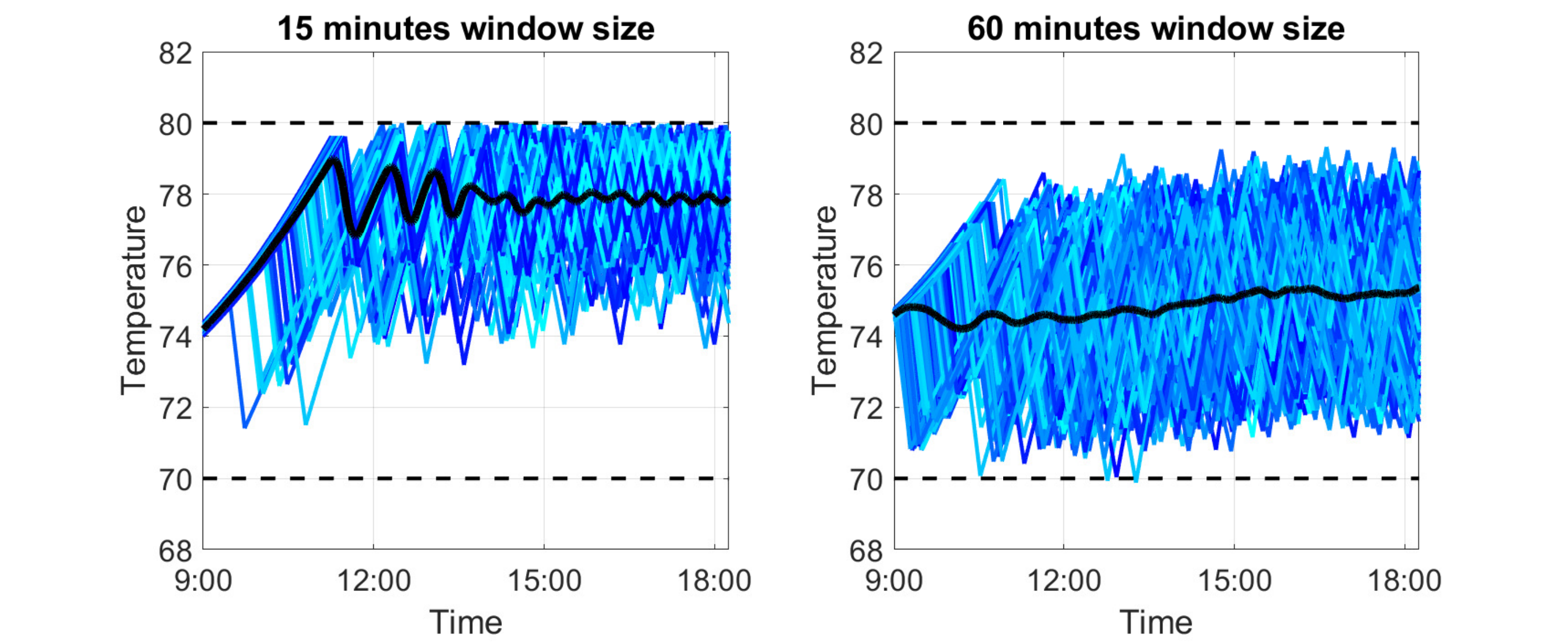}
\caption{Temperatures of 375 households under different window sizes: 15 minutes window size (left) and 60 minutes window size (right).}\label{fig:375ac_w}
\end{figure}

\subsubsection{Voltage Regulation}
Lastly, we plot the voltage in Fig.~\ref{fig:controlv}. Due to coordinated efforts of all controllable devices, the voltage magnitude is brought within the {original 1.05 p.u. upper bound}, compared with uncontrolled one. {Meanwhile, we have observed voltage above the enforced 1.04 p.u. upper bounds. This is mainly due to two reasons: i) the randomly realized power injections cause random fluctuation in voltages, and ii) the online algorithm may not enforce feasibility at each update. Nevertheless, the robust design guarantees that the original constraints are satisfied by large margin.}

\section{Conclusion}\label{sec:conclusion}
We have proposed a distributed stochastic dual algorithm for managing DERs with both continuous and discrete decision variables as well as device dynamics, and extended it to the practical realtime setting with time-varying operating conditions, asynchronous updates by devices, and feedback being leveraged to account for nonlinear power flows as well as reduce communication overhead. The resulting algorithm provides a general online stochastic optimization algorithm for coordinating networked DERs to meet operational and economic objectives and constraints.  We characterize the convergence of the algorithm analytically and evaluate its performance numerically.

As further work, we will exploit the tree structure of the distribution system and explore efficient data structure for practical implementation of the proposed algorithm, with the goal of achieving seconds or sub-second computation time to meet the real-time control requirement for large distribution systems with up to one hundred thousands nodes/devices. We will also investigate further robust design with statistical performance guarantee, as well as apply the proposed stochastic optimization approach to the design and control of other networked systems such as those in communication networks that involve discrete decision variables.



\section*{Acknowledgement}
We thank the anonymous reviewers for their constructive suggestions to help improve this work, and Dr. Changhong Zhao for helpful discussion.

{
This work was authored in part by the National Renewable Energy Laboratory, operated by Alliance for Sustainable Energy, LLC, for the U.S. Department of Energy (DOE). Funding provided by Grid Modernization Lab Consortium (GMLC). The views expressed in the article do not necessarily represent the views of the DOE or the U.S. Government. The U.S. Government retains and the publisher, by accepting the article for publication, acknowledges that the U.S. Government retains a nonexclusive, paid-up, irrevocable, worldwide license to publish or reproduce the published form of this work, or allow others to do so, for U.S. Government purposes.
}
\bibliographystyle{IEEEtran}
\bibliography{biblio.bib}

{
\appendix

\subsection{Characterization of $\Delta$ in \eqref{eq:Delta}}
\label{app:Delta}
This part characterizes the bound $\Delta$ of $\text{Var}\big(\bm{g}^t(\bm{z}^t)\big)$ defined in \eqref{eq:Delta}. For illustrative purpose, we only consider the two-point probability distribution in \eqref{eq:probability} for all discrete devices; extensions can be made in a straightforward way for more general probability distribution.  Because the randomized step only takes place for the current time $t$, we have that $\text{Var}\big(\bm{g}^t(\bm{z}^t)\big)=\text{Var}\big(g^t(z^t)\big)$. For notational simplicity, we use $g^t(z^t)=R^tz^t+g_0^t$ with $R^t\in\mathbb{R}^{M \times N}$ and a constant vector $g_0^t$. Then the variance can be written as:
\begin{eqnarray}
&&\text{Var}(g^t(z^t))\nonumber\\
&=&E[\|g^t(z^t)-g^t(z^{t*})\|^2]-\|E[g^t(z^t)-g^t(z^{t*})]\|^2\nonumber\\
&=&E\Big[\sum_{i=1}^M\Big(\sum_{j\in\cN}R^t_{ij}(z^t_j-z_j^{t*})\Big)^2\Big]\nonumber\\[-4pt]
&\leq&\sum_{i=1}^M\Big(\sum_{j\in\cN}R^{t2}_{ij}\cdot E\big[\sum_{j\in\cN}(z^t_j-z_j^{t*})^2\big]\Big)\nonumber\\[-3pt]
&\leq&\sum_{i=1}^M\Big(\sum_{j\in\cN}R^{t2}_{ij} \cdot \sum_{j\in\cN}\sum_{d\in\cD_{S_j}} E\big[(z^t_{j,d}-z_{j,d}^{t*})^2\big]\Big)\nonumber[-3pt\\
&=& \sum_{i=1}^M\Big(\sum_{j\in\cN}R^{t2}_{ij} \cdot  \sum_{j\in\cN}\sum_{d\in\cD_{S_j}} (p_{j,d}^*-\underline{p}_{j,d})(\overline{p}_{j,d}-p_{j,d}^*)\Big)\nonumber
\end{eqnarray}
\begin{eqnarray}
&\leq& \sum_{i=1}^M\Big( \sum_{j\in\cN}R^{t2}_{ij} \cdot \sum_{j\in\cN}\sum_{d\in\cD_{S_j}}{(\overline{p}_{j,d}-\underline{p}_{j,d})^2}/{4}\Big)\nonumber\\[-4pt]
&\leq& M \cdot |\cD_{S}|\cdot \|R^t\|_F^2\cdot \underset{j,d}{\max} (\overline{p}_{j,d}-\underline{p}_{j,d})^2/4,\label{eq:deltavalue}
\end{eqnarray}
where we apply Cauchy-Schwarz inequality in the first inequality, Jensen's inequality in the second, and the second equality is based on the probability distribution (\ref{eq:probability}) with real power only.

Indeed, \eqref{eq:deltavalue} can serve as a  bound for $\Delta$, which is increasing with the number of discrete devices $|\cD_{S}|$ and the control ``granularity'' of the device $ \overline{p}_{j,d}-\underline{p}_{j,d}$.


\subsection{Robust Design for Voltage Regulation}\label{app:robustvolt}

With our proposed stochastic algorithm, the operational constraints such as voltage-regulation constraints may possibly be violated. {Using} the voltage-regulation as an illustrative example, we now discuss the problem of robust design that ensures that the probability of voltage violation is below certain specified value. 

Specifically, we consider the voltage constraint $\underline{v}_i\leq v_i\leq \overline{v}_i$, and study the design of tighter bounds $\underline{v}'_i,~ \overline{v}'_i$ with $\underline{v}_i< \underline{v}'_i\leq v_i \leq \overline{v}'_i< \overline{v}_i$ such that the voltages are within $[\underline{v}_i,~ \overline{v}_i]$  with large enough probability under the proposed algorithm. 

\begin{proposition}\label{pro:rob}
	Given $\delta>0$, if the  voltage upper and lower bounds are set as:
	$\overline{v}'_i \leq \overline{v}_i-\delta$ and $\underline{v}'_i \geq \underline{v}_i+\delta$,
	then
	\begin{subequations}\label{eq:robustvoltage}
	\begin{eqnarray}
		\text{Prob}(v_i\geq \overline{v}_i)\leq{\text{Var} (v_i)}/{2\delta^2},\\ \text{Prob}(v_i\leq \underline{v}_i)\leq{\text{Var}(v_i)}/{2\delta^2}.
	\end{eqnarray}
	\end{subequations}

\end{proposition}
\begin{proof}
Let $\tilde{v}_i=E[v_i]$. By Chebyshev's inequality \cite{stark2014probability}, given $\delta>0$, we have
\begin{eqnarray}
\text{Prob}(|v_i-\tilde{v}_i|\geq \delta)\leq \frac{\text{Var}(v_i)}{\delta^2}.\nonumber
\end{eqnarray}

Consider first the upper bound. Choose a ``robust'' bound $\overline{v}'_i \leq \overline{v}_i-\delta$, and we have $\tilde{v}_i\leq \overline{v}'_i$. 
The probability of voltage violation is as follows:
\begin{eqnarray}
&&\text{Prob}(v_i\geq \overline{v}_i)=\text{Prob}(v_i-\tilde{v}_i\geq \overline{v}_i-\tilde{v}_i)\nonumber\\
&\leq&\text{Prob}(v_i-\tilde{v}_i\geq \overline{v}_i-\overline{v}'_i)\leq\text{Prob}(v_i-\tilde{v}_i\geq \delta)\nonumber\\
&=&\frac{1}{2}\cdot\text{Prob}(|v_i-\tilde{v}_i|\geq \delta)\leq \frac{\text{Var}(v_i)}{2\delta^2}.\nonumber
\end{eqnarray}

The lower bound can be handled similarly. 
This concludes the proof.
\end{proof}

By \eqref{eq:deltavalue}--\eqref{eq:robustvoltage}, we compute the robust voltage bounds for the simulated scenario in Section~\ref{sec:simulation}. We get $\underline{v}'_i=0.965~\text{p.u.},\ \overline{v}'_i=1.035~\text{p.u.}$ with  $\text{Var}(v_i)/\delta^2\leq 5\%$, given the voltage regulation bounds $\underline{v}_i=0.95~\text{p.u.},\ \overline{v}'_i=1.05~\text{p.u.}$ In other words, we expect less than 5\% chance of voltage violation if we use the tighter bounds in the proposed algorithm. 

Notice that \eqref{eq:deltavalue}--\eqref{eq:robustvoltage} provide a conservative estimate, and the actual probability of voltage violation may be much smaller. 
}
\end{document}

%% file: introduction_v2.tex
\section{Introduction}

Power grids are experiencing an increasing flexibility in control on both supply and demand sides thanks to growing penetration of distributed energy resources (DERs) on the distribution level, including roof-top photovoltaic (PV) panels, electric vehicles (EV), smart batteries, thermostatically controlled loads (TCLs, e.g., water heaters and air-conditioners (A/Cs)), and other responsive loads, etc. While these flexible or controllable devices  can potentially provide ancillary services for the grid \cite{dall2017unlocking}, coordinating a large number of such devices with various dynamics and constraints to achieve network-wide objectives such as voltage regulation, frequency control, and economic efficiency is extremely challenging. Moreover, unlike the traditional assets owned by utility companies, the mass customer-owned devices are not necessarily subject to the control of network operators unless properly incentivized.  This calls for the joint design of distributed control and incentive/market mechanisms to bring self-interested customers into the control loop so that network-wide objectives and constraints can be achieved by inducing the desired customer behaviors through proper incentives.


Indeed, there is a lot of effort on market-based control algorithms for tapping and coordinating the customer-owned DERs; see, e.g., \cite{mohsenian2010autonomous, Maharjan13, Pedram14,li2016market} for demand management, \cite{Liu08,zhou2017incentive2} for voltage regulation, and~\cite{vrettos2016robust} for frequency regulation. 
In particular, in \cite{zhou2017incentive2} we consider a social welfare optimization problem that captures the operational and economic objectives of both network operator and customers as well as the voltage constraints, and design an online optimization framework based on a primal-dual gradient algorithm such that the network operator and customers pursue the given operational and economic objectives while concurrently ensuring that the voltages are within the prescribed limits. There are, however, a few important limitations with the existing work; notably, discreteness in decision variables for certain devices such as A/Cs and batteries is mostly ignored. 
Discrete decision variables make the problem non-convex, for which in general no efficient algorithms exist. In this paper, we propose a practical stochastic optimization approach to address the problem of discrete decision variables. 

Specifically, we formulate a general multi-period social welfare maximization problem (minimum cost problem, in terms of minimization) with both continuous and discrete decision variables as well as device dynamics for managing DERs, and introduce a convex relaxation of the problem by replacing discrete feasible sets with their convex hulls. We then propose a distributed stochastic algorithm that recovers discrete decision variables randomly according to convex combination coefficients from the dual gradient algorithm for the relaxed problem.  
The resulting algorithm is distributed, where the self-interested customers update their devices' power setpoints based on local constraints and individual cost functions for given incentive signals and the network operator updates the incentive signals based on network-wide operational constraints. 



The convergence of the dual gradient method can be hard to characterize, as the dual function may be non-smooth and not strongly concave.  
The additional stochasticity of our algorithm makes it more challenging. In literature diminishing stepsizes are usually necessary for convergence (see,  e.g., \cite{ram2009incremental,zhu2012distributed,polyak1987introduction,boyd2006subgradient,zhou2017stochastic}), which may not be practical in, e.g., a real-time and/or asynchronous setting. Nevertheless, we leverage recent insights in the dual method to characterize the convergence of the proposed stochastic dual algorithm with {constant} stepsizes. 
To the best of our knowledge, this is a first convergence characterization of its kind for the dual method applied in power systems. 

Notice that, due to the intermittent {renewable} generation and the uncontrollable loads, the operating condition of power grid may change at a fast timescale, which allows only a few iterations of the above-mentioned algorithm. Moreover, different DER devices may be featured with different timescales in control; e.g., devices with continuous decision variables can update power setpoints at relatively fast timescales, while those with discrete decision variables may only be able to update at relatively slow timescales. We therefore extend the proposed stochastic dual algorithm to a practical online realtime setting where 1) during each timeslot the algorithm can only run one or a few iterations in order to track the time-varying ``optimal'' operating point and 2) devices may update power setpoints asynchronously at different times. Also notice that, while we use a linearized power flow model to guide tractable algorithm design, our realtime algorithm will leverage the measured values of relevant electrical variables on the power system to account for the nonlinear power flows as well as reduce communication overhead.  The resulting real-time feedback receding horizon control (RHC) algorithm provides a general online stochastic optimization algorithm for coordinating networked DERs with discrete power setpoints and dynamics to meet operational and economic objectives and constraints. We further characterize its convergence analytically and evaluate its performance numerically.

\subsection{Related Work}
We now briefly review some related work besides those already mentioned in the above. 


\paragraph{Discrete Decision Variables}
Algorithms designed for discrete decision variables in power systems roughly fall into two categories: deterministic but usually generate suboptimal solutions, see, e.g., \cite{bernstein2015design, liu1992discrete}; and stochastic but often lack rigorous performance characterization, see, e.g., \cite{kim2013scalable,macfie2010proposed}. Some of (centralized) deterministic algorithms, {e.g., \cite{ashraphijuo2016strong},}  rely on (commercial) solvers which may not be scalable to very large systems to meet the requirement of realtime control. 

\paragraph{Controlling Devices with Dynamics}
Controllable devices with dynamics are usually handled in two ways: control based on heuristic or engineering intuition, see, e.g., \cite{mathieu2013state} that controls TCLs based on temperature status; and optimization-based control that can integrate specific objective functions and constraints, see, e.g., \cite{li2011optimal} that considers device dynamics but does not involve discrete variables, \cite{tsui2012demand} that considers device dynamics and discrete decision variables but employs commercial solvers,  {\cite{kraning2014dynamic,mhanna2016fast} that solve OPF over various devices with dynamics but consider equality constraints  only, and \cite{zamzam2018optimal, anjos2019decentralized} that formulate mixed-integer linear programs which cannot be applied to solving more general problems with convex cost functions. }

\paragraph{Market-Based/Demand-Response Design}
Existing literature on market-based and demand-response problem formulation mostly focuses on demand/supply balancing, without considering network structure, see, e.g.,\cite{mohsenian2010autonomous,Pedram14,Maharjan13,Tushar14,li2016market,li2011optimal,tsui2012demand}. On the other hand, those that are network-cognizant often require complex communication schemes and do not leverage feedback to reduce communication, see, e.g., \cite{li2012demand,LinaCDC15}. 

\paragraph{Distributed Optimization and Control of {Distribution} Grids} While our modeling framework and algorithm apply to general operational constraints, one of the major applications is distributed voltage regulation that has recently drawn considerable research; see, e.g.,  \cite{zhou2018reverse,simpson2017voltage,baker2017network, dall2017optimal}. 
Online optimization with feedback control is also studied in, e.g., \cite{hauswirth2017online, dallanese2016optimal, gan2016online, bolognani2015distributed} to deal with nonlinear power flow models. 
We also refer to \cite{molzahn2017survey} and references therein for a recent survey on distributed optimization and control of power systems.

{
	\paragraph{Comparison with~\cite{zhou2017incentive2}}
	In \cite{zhou2017incentive2} we have proposed a real-time voltage regulation framework, based on a primal-dual gradient algorithm. The present work significantly extends \cite{zhou2017incentive2} in the following aspects: 1) We consider  a more general model with a broader class of controllable  assets, including DERs with discrete decision variables and with temporally-correlated constraints. We introduce  a randomized step to efficiently cope  with discrete decision variables and a receding horizon control (RHC) method to handle the multi-period setting. 2) We employ a dual algorithm, which exhibits better convergence property than the primal-dual gradient algorithm proposed in \cite{zhou2017incentive2}. 3) We consider and analyze a more realistic scenario where the algorithmic updates are  asynchronous. 4) The simulations include a considerably larger number of  DERs than \cite{zhou2017incentive2} (768 devices v.s. 18 devices). }


%
The rest of the paper is organized as follows. Section~\ref{sec:model} presents power network model along with device models, and formulates the social welfare optimization problem. Section~\ref{sec:off} presents convex relaxation and an offline distributed stochastic dual algorithm for solving the problem. Section~\ref{sec:online} extends the algorithm to a practical realtime setting with time-varying operating conditions, asynchronous updates, and feedback. Section~\ref{sec:simulation} provides numerical examples to evaluate the performance of the algorithm, and Section~\ref{sec:conclusion} concludes the paper. 

\begin{table}
\begin{center}
\begin{tabular}{|ll|}   
	\hline 
        $\cN$ & Set of nodes, excluding node $0$; $\cN:=\{1, ..., N\}$\\
        $\cE$ & Set of distribution lines\\
	$\cD_i$ & Controllable DERs at node $i$\\ 
	$\cD^t_i$ & Controllable DERs at node $i$ that update setpoint at time $t$\\ 
        $p^t_{i,0},q^t_{i,0}$ & Non-controllable power injection of node $i$ at time $t$\\
        $p^t_{i,d},q^t_{i,d}$ & Power injection from DER $d$ at node $i$ at time $t$\\ 
	$\mathcal{Z}^t_{i,d}$ & Feasible set of DERs $d$ at node $i$ \\
        $\mathcal{Z}^t_i$ & $\mathcal{Z}_i^t=\bigtimes_{d\in\cD_i}\mathcal{Z}^t_{i,d}$\\
        $\mathcal{Z}^t$ &$\mathcal{Z}^t=\bigtimes_{i\in\hN}\mathcal{Z}^t_i$\\
        $z^t_i$ &  $z^t_i = \{p^t_{i,d},q^t_{i,d}\}_{d\in\cD_i}\in\mathcal{Z}^t_i$\\
        $z^t$ & $z^t = \{p^t_{i,d},q^t_{i,d}\}_{d\in\cD_i}^{i\in\cN}\in\mathcal{Z}^t$\\
        $x^t_{i,d}$& Status of device $d$ from node $i$ at time $t$\\
        $x_i^t$ & $x_i^t=\{x_{i,d}^t\}_{d\in\cD_i}$\\
        $x^t$ & $x^t=\{x_{i,d}^t\}^{i\in\cN}_{d\in\cD_i}$\\
        $p^t_i,q^t_i$ & Aggregated power injections of node $i$ at time $t$,\\
	& $p_{i}^t=p_{i,0}^{t}+\sum_{d\in\cD_i}p^t_{i,d}$, $q^t_{i} =q_{i,0}^{t}+\sum_{d\in\cD_i}q^t_{i,d}$ \\
	$p^t,q^t$ &  $p^t := [p^t_1, \ldots, p^t_N]^{\top}$, $q^t := [q^t_1, \ldots, q^t_N]^{\top}$ \\
        $y_i^t, \hat{y}_i^t$  &  Actual and linearized system states of node $i$ at time $t$\\
        $y^t, \hat{y}^t$ & $y^t := [y^t_1, \ldots, y^t_N]^{\top}$, $\hat{y}^t := [\hat{y}^t_1, \ldots, \hat{y}^t_N]^{\top}$ \\
        $\alpha_i^t,\beta_i^t$ & Signal for real and reactive power of node $i$ at time $t$\\
        $\alpha^t,\beta^t$ & $\alpha^t := [\alpha^t_1, \ldots, \alpha^t_N]^{\top}$, $\beta^t := [\beta^t_1, \ldots, \beta^t_N]^{\top}$ \\
        $W_t^w$ & Time window $\{t, t+1, \cdots, t+w\}$\\
	$A,B,c$ & System states linearization parameters\\
	$I$& Identity matrix\\
	$\bm{1}$ & Vector collecting all ones\\
        \hline 
\end{tabular}
\caption{Notation. For notational simplicity, we apply \textbf{bold} symbols with superscript $t$ to represent the stacked variables within the $w$ timeslots starting from time $t$, e.g., $\bm{z}^{t}=[(z^{t})^\top,\ldots,(z^{t+w})^\top]^{\top}$. Similarly, functions denoted with a bold letter are vector-valued function over $w$ time steps. Extra subscript $t$ in online algorithm denotes the time instants when the variables are updated. 
$\|\cdot\|_F$ denotes the Frobenius norm, and $\|\cdot\|$ without subscript denotes $L^2$-norm.}\label{tab:not} 
\end{center}
\end{table}